\documentclass{amsart}
\usepackage{amsmath,amssymb,mathrsfs,amscd}
\usepackage[left=3.5cm,right=3.5cm]{geometry}

\usepackage[all]{xy}

\usepackage{wasysym}
\usepackage{xcolor}
\usepackage[normalem]{ulem}
\usepackage{comment}
\usepackage{mathtools}
\usepackage{mathbbol}

\theoremstyle{plain}
\newtheorem{Theorem}{Theorem}[section]
\newtheorem{Definition}[Theorem]{Definition}
\newtheorem{Lemma}[Theorem]{Lemma}
\newtheorem{Proposition}[Theorem]{Proposition}
\newtheorem{Corollary}[Theorem]{Corollary}
\newtheorem*{Remark}{Remark}
\newtheorem{Examples}[Theorem]{Examples}
\newtheorem{Non-examples}[Theorem]{Non-examples}
\newtheorem{Convention}[Theorem]{Convention}

\newcommand{\ax}{\mathbf{Ax}}

\usepackage[bookmarks=true, pdfauthor={}]{hyperref}

\begin{document}
\title{GENERIC FREE SUBGROUPS AND STATISTICAL HYPERBOLICITY}

\author{Suzhen Han}

\author{Wen-yuan Yang}
\address{Beijing International Center for Mathematical Research (BICMR), Beijing University, No. 5 Yiheyuan Road, Haidian District, Beijing, China}
\email{suzhenhan@pku.edu.cn}
\email{yabziz@gmail.com}

\thanks{W. Y. is supported by the National Natural Science Foundation of China (No. 11771022).}

\begin{abstract}
This paper studies the generic behavior of  $k$-tuple elements for $k\ge 2$ in a proper group action with contracting elements, with applications towards relatively hyperbolic groups, CAT(0) groups and mapping class groups. For a class of  statistically convex-cocompact action, we show that an exponential generic set of $k$ elements for any fixed $k\ge 2$ generates a quasi-isometrically embedded free subgroup of rank $k$. For $k=2$, we study the sprawl property of group actions and  establish that  the class of statistically convex-cocompact actions is statistically hyperbolic in a sense of M. Duchin, S. Leli\`{e}vre, and C. Mooney.

For any proper action with a contracting element,  if it satisfies a condition introduced by Dal'bo-Otal-Peign\'{e} and has purely exponential growth,  we obtain the same results on generic free subgroups and statistical hyperbolicity.
\end{abstract}

\maketitle

\section{Introduction}
\subsection{Motivation and background}

Suppose that a group $G$ admits a proper and isometric action on a proper geodesic metric space $(Y,d)$. The group $G$ is assumed to be \textit{non-elementary}: it is not  virtually cyclic.   An element $g\in G$  is called \textit{contracting} if for some basepoint $o\in Y$, an orbit $\{h^n\cdot o:n\in \mathbb{Z}\}$ is a contracting subset, and the map $n\in \mathbb{Z}\mapsto h^no\in G$ is a quasi-isometric embedding. Here a subset $X$ is called \textit{contracting} if any metric ball disjoint with $X$ has a uniformly bounded projection to $X$ (see \cite{M,BF}). It is clear that this definition does not depend on the choice of the basepoint.

The prototype of a contracting element is a hyperbolic isometry on Gromov-hyperbolic spaces, but more interesting examples are furnished by the following:
\begin{itemize}
\item hyperbolic elements in relatively hyperbolic groups or groups with nontrivial Floyd boundary (see \cite{GP,GP1});
\item rank-1 elements in CAT(0) groups (see \cite{B,BF});
\item certain infinite order elements in certain small cancellation groups (see \cite{ACGH});
\item pseudo-Anosov elements in mapping class groups of closed oriented surfaces with genus greater than two acting on Teichm\"{u}ller space (see \cite{M}).
\end{itemize}

In \cite{Y}, the second-named author proved that, for a class of \textit{statistically convex-cocompact} actions defined below,  the set $X$ of contracting elements is \textit{exponentially generic} in the ball model:
\[ \frac{|X\cap B_n|}{|B_n|} \to 1\]
exponentially fast, where $B_n:=\{g\in G: d(o, go)\leq n\}$.

Along this line, the goal of this paper is to  continue the  study of generic properties for $k$-tuples of elements in $G$ for a fixed $k\ge 2$. To that end, we introduce a few more notations. We   fix a basepoint $o\in Y$ and  denote $|g|=  d(o,go) $ for easy notation. Denote $G^{(k)}=\{(u_1,\cdots,u_k):u_i\in G\}$ and $B_n^{(k)}=\{(u_1,\cdots,u_k)\in G^{(k)}:|u_i|\leq n\}.$ When $k$ is understood, we write $\overrightarrow{u}$ for $(u_1,\cdots,u_k)$, and $|\overrightarrow{u}|$ for $\max\{|u_i|:1\leq i\leq k\}$.

The \textit{asymptotic density} of a subset $X\subseteq G^{(k)}$ in ball model is defined as
\[\mu(X)=\lim_{n\rightarrow\infty}\frac{|X\cap B_n^{(k)}|}{|B_n^{(k)}|}\]
if the limit exists. If the convergence happens exponentially fast, we denote $\mu(X) \stackrel{exp}{=}\lambda\in [0, 1]$.
We shall be interested in the extreme cases   $\mu(X)\stackrel{exp}{=}1$ (resp. $\mu(X)=1$) which are called   \textit{exponentially generic} (resp. \textit{generic}). By definition, the complement of an (exponentially) generic set is called \textit{(exponentially) negligible}.

The generic property of $k$-tuple of elements has been studied using random walks in various class of groups with negative curvature. Let $\mu$ be a probability measure with finite support on the group $G$ so that the support generates $G$ as a semi-group. A $\mu$-random walk is a product of a sequence of independent identical $\mu$-distributed random variables on $G$. In our setting, Sisto \cite{Sisto} proved that the $n$-th step of a simple random rank lands on a contracting element with asymptotic probability one. In mapping class groups, this was obtained by Maher for pseudo-Anosov elements, and the most general result is, as far as we know,  due to Maher and Tiozzo \cite{MaherT} for any non-elementary action on a hyperbolic space where  random elements are being loxodromic. When $k\ge 2$,  Gilman, Miasnikov, and Osin \cite{GMO} proved in hyperbolic groups that two simple random walks on the Cayley graph  stay at a ping-pong position in $n$-steps with asymptotic probability one so that they generate an undistorted free group of rank 2.  The same result holds in non-virtually solvable linear groups \cite{Aoun} and in  mapping class groups \cite{Rivin2, TiozzoT, MaherS} for two independent $\mu$-random walks. In fact, most of these works  are stated in a general class of groups with hyperbolic embedded subgroups called b  Dahmani, Guirardel and Osin \cite{DGO} and equivalently, the class of acylindrical hyperbolic groups in the sense of Osin \cite{Osin}.  It is worth pointing out that a proper action with a contracting element  is  acylindrical hyperbolic by a result of Sisto \cite{Sisto}. However, our first goal is to address the analogue of generic free subgroups using counting measure as above instead of probability measure from random walks.

In fact, studying the generic properties of $k$-tuple elements in a counting measure is not a new idea. In \cite{DLM}, M. Duchin, S. Leli\`{e}vre, and C. Mooney initiated a study of sprawl property of   pair of points in the space. The notion of statistical hyperbolicity is then introduced to capture  negative curvature  in a statistical sense.  Roughly speaking, the intuitive meaning could be explained as follows: consider the annular set $$A(n,\Delta)=\{g\in G:||g|-n|\leq\Delta\}$$ for $\Delta>0$. On average, a random pair of points $x,y$ on an annular set $A(n,\Delta)$ of the group has the distance $d(xo,yo)$ nearly equal to $2n$. We formulize this concept using both annuli and balls.
\begin{Definition}\label{DefnStaHyp}
Let $G$ admit a proper action on a geodesic metric space $(Y,d)$. Define
\[E_B(G)=\lim_{n\rightarrow+\infty}\frac{1}{|B_n|^2}\sum_{x,y\in B_n}\frac{d(x,y)}{n},\]
and for a constant $\Delta>0$,
\[E_A(G, \Delta)=\lim_{n\rightarrow+\infty}\frac{1}{|A(n, \Delta)|^2}\sum_{x,y\in A(n,\Delta)}\frac{d(x,y)}{n},\]
if the limit exists.  The action  is called \textit{statistically hyperbolic} in annuli (resp. in balls) if $E_A(G, \Delta)=2$ for any sufficiently large $\Delta>0$ (resp. $E_B(G)=2$).
\end{Definition}
\begin{Remark}
In \cite{DLM} this definition was introduced using annular model with $\Delta=0$  in  the Cayley graph of groups. Here we consider also the quantity $E_B(G)$ without involving the extra parameter $\Delta$. In our results, we obtain   $E_A(G, \Delta)=E_B(G)=2$ along the same line of proofs.
\end{Remark}

The non-examples  include elementary groups, $\mathbb{Z}^d$ for $d\geq2$, and the integer Heisenberg group for any finite generating set among the others (cf. \cite{DLM}). In the opposite, the exact value of $E_B(G)=2$ indeed happens for many groups with certain negative curvature from a  point of view of coarse geometry. For instance,  non-elementary relatively hyperbolic groups are statistical hyperbolic for any finite generating set (cf. \cite{DLM,OY}). Moreover,  the statistical hyperbolicity is preserved under certain direct product of a   relatively hyperbolic group and a group. And the lamplighter groups $\mathbb{Z}_m\wr\mathbb{Z}$  where $m\geq 2$  are statistical hyperbolic for certain generating sets \cite{DLM}.

The notion of statistical hyperbolicity could be considered for any metric space with a measure as in  \cite{DLM}, rather than our definition using a counting measure. In this direction, it was proved in the same paper    that for any $m,p\geq 2$, the Diestel-Leader graph $DL(m,p)$ is statistically hyperbolic. The statistical hyperbolicity for Teichm\"{u}ller space with various measures was proved by Dowdall, Duchin and Masur  in \cite{DDM}.

The second goal of the paper is to generalize these results in a very general class of proper actions using counting measures from orbits in Definition \ref{DefnStaHyp}. In what follows, we shall describe our results in detail.

\subsection{Main results}
In order to expose our results, we   first give a quick overview of the various classes of actions under consideration in this study.  First of all, we consider the class of  statistically convex-cocompact actions introduced in \cite{Y} which generalizes a convex-cocompact action in a statistical sense. Making this idea precise requires a notion of \textit{growth rate} of a subset $X$ in $G$:
\[\delta_X=\limsup_{n\rightarrow\infty}\frac{\ln|X\cap B_n|}{n}.\] It is clear that the value $\delta_X$ does not depend on the choice of the basepoint.
By abuse of language, a geodesic between two sets $A$ and $B$ is a geodesic between $a\in A$ and $b\in B$.

Given constants $0\leq M_1\leq M_2$, let $\mathcal{O}_{M_1,M_2}$ be the set of element $g\in G$ such that there exists some geodesic $\gamma$ between $B(o,M_2)$ and $B(go,M_2)$ with the property that the interior of $\gamma$ lie outside $N_{M_1}(Go)$.

\begin{Definition}[SCC Action]\label{SCCDefn}
If there exist positive constants $M_1,M_2>0$ such that $\delta_{\mathcal{O}_{M_1,M_2}}<\delta_G<\infty$, then the proper action of $G$ on $Y$ is called \textit{statistically convex-cocompact} (SCC).
\end{Definition}

The idea to define the set $\mathcal{O}_{M_1,M_2}$ is to look at the action of  the fundamental group of a finite volume Hadamard manifold on its universal cover. It is then easy to see that for appropriate constants $M_1, M_2>0$, the set $\mathcal{O}_{M_1,M_2}$ coincides with  the union of cusp subgroups up to a  finite Hausdorff distance. The assumption in SCC actions was called a \textit{parabolic gap condition} by Dal'bo, Otal and Peign\'{e} in \cite{DOP}. One of motivations of this study is to push forward the analogy between the concave set $\mathcal{O}_{M_1,M_2}$ and the (union of) parabolic cusp regions. This allows us to draw conclusions for   the SCC actions through the analogy with the geometrically finite actions, which have been well studied in last twenty years.

Moreover, our study suggests considering a class of proper actions satisfying a more general condition introduced at the same paper  \cite{DOP}. The condition, reformulated below, is proved   to be equivalent to the finiteness of Bowen-Margulis-Sullivan (BMS) measure on the geodesic flow of the unit tangent bundle of a geometrically finite  Hadamard manifold in \cite{DOP}, and later for any Hadamard manifold by Pit and Shapira \cite[Theorem 2]{PS18}.

\begin{Definition}[DOP condition]\label{DOPDef}
The group action of $G$ on $Y$ satisfies the \textit{Dal'bo-Otal-Peign\'{e} (DOP) condition} if there exist two positive constants $M_1,M_2>0$ such that
\[\sum_{g\in \mathcal{O}_{M_1,M_2}}|g|\exp(-\delta_G|g|)<\infty\]

\end{Definition}
\begin{Remark}
We remark that, in the setting of negatively curved manifolds,  the DOP condition is called \textit{positive recurrent} by Pit and Shapira in \cite{PS18}, whereas the notion of SCC actions is called \textit{strongly positive recurrent}   by Shapira and Tapie in \cite{ST18}.  We thank R\'emi Coulon to bring these references to our attention.
\end{Remark}

The concept of the geodesic flow is non-applicable in a general geodesic metric space with negative curvature such as contracting property. However, the definition of the DOP condition could be always made, and so could be understood as substitute of finite BMS measures in a general metric space. One of Roblin's results \cite[Th\'eoreme 4.1]{Roblin} stated in the setting of a geometrically finite Hadamard manifold  characterized the finiteness of BMS measures by a \textit{purely exponentially growth} (PEG)  of the action:
\[|B_n| \asymp\exp(\delta_Gn).\]
Hence, the class of proper actions with purely exponential growth should be viewed as equivalents of DOP conditions. We expect this relation persists in a very general setting,  and remark that it is indeed true for the class of geometrically finite action on a $\delta$-hyperbolic space in \cite{YANG8} (weaker than the setting of Roblin).

Our first main result establishes that generic $k$-tuple elements are the free basis of a free group with quasi-isometrically embedded property for the above two class of actions.

\begin{Theorem}\label{thm:gen}
Assume that    a non-elementary group  $G$  acts properly on a geodesic metric space $(Y,d)$ with a contracting element. If $G$  satisfies the DOP condition and has purely exponentially growth. Then for any $k\ge 2$,  the set of all tuples $(u_1,\cdots,u_k)\in G^{(k)}$ so that $u_1,\cdots,u_k$ generate a free subgroup of rank $k$ in $G$ is generic in $G^{(k)}$. Moreover, these free subgroups are quasi-isometrically embedded with contracting images.
\end{Theorem}

When the action is SCC, the above assumptions hold, and moreover, we can obtain an exponential convergence rate  for the above conclusion.

\begin{Theorem}\label{thm:gen2}
Assume that a non-elementary group $G$ admit a SCC action on a geodesic metric space $(Y,d)$ with a contracting element. Then for any $k\ge 2$,  the set of all $(u_1,\cdots,u_k)\in G^{(k)}$ for which $u_1,\cdots,u_k$ generate a    free subgroup of rank $k$ in $G$ is exponentially generic in $G^{(k)}$. Moreover, these free subgroups are quasi-isometrically embedded with contracting images.
\end{Theorem}

A   group  generated by  a finite set   acts cocompactly on its Cayley graph, so our results apply for this particular case. A finitely generated subgroup $H$ is called \textit{undistorted} if the inclusion $H\subset G$ is quasi-isometrically embedded with respect to word metrics.

\begin{Corollary}\label{CorCayleyGraph}
Let $G$ be a   non-elementary group with a finite generating set $S$. If $G$ has a contracting element, then the set of all $(u_1,\cdots,u_k)\in G^{(k)}$ for which $u_1,\cdots,u_k$ generate an undistorted  free subgroup of rank $k$ in $G$ is exponentially generic in $G^{(k)}$.
\end{Corollary}

To illustrate consequences of previous results, we remark that the following examples of groups  with contracting elements with respect to the Cayley graph:
\begin{enumerate}
\item any relatively hyperbolic group $G$ acts on a Cayley graph $\mathscr{G}(G,S)$ with respect to a finite generating set $S$. See \cite{GP1}.
\item any group $G$ with non-trivial Floyd boundary acts on a Cayley graph $\mathscr{G}(G,S)$ with respect to a finite generating set $S$. \cite{GP1}.
\item the right-angled Artin (Coxeter) groups  with respect to the standard generating set, if they are not virtually direct product. \cite{BehC, BHS, ChaSul}.
\item the Gr'$(\frac{1}{6})$-labeled graphical small cancellation group $G$ with finite components labeled by a finite set $S$ acts on the Cayley graph $\mathscr{G}(G,S)$. See \cite{ACGH}.
\end{enumerate}
Thus, by Corollary \ref{CorCayleyGraph}, the list of these   examples all have the generic free basis property.  We remark that this result is even new in the class of relatively hyperbolic groups.

We next explain an application of Theorem \ref{thm:gen2} about surface group extensions.  Let $\mathbb{Mod}(\Sigma_g)$ be the mapping class group  of a closed oriented surface $\Sigma_g$ of genus $g\ge 2$. Combining the results of Minsky \cite{M} and Eskin-Mirzakhani-Rafi \cite{EMR} we know that the action of $\mathbb{Mod}(\Sigma_g)$  on Teichm\"{u}ller space $\mathcal T(\Sigma_g)$ is a SCC action with a contracting element. By Theorem \ref{thm:gen2},  we obtain the exponential genericity of   $k$-tuple elements $(u_1, u_2, \cdots, u_k)$ being free basis  in the counting measure from Teichm\"{u}ller metric. Denote $\Gamma:=\langle u_1, u_2, \cdots, u_k \rangle$.  Marking a point $p\in \Sigma_g$, the Bireman exact sequence  in \cite{Birman} gives an extension $E_\Gamma$ in $\mathbb{Mod}(\Sigma_g, p)$ of the surface group   $\pi_1(\Sigma_g, p)$ by $\Gamma$  as follows
\[1\to \pi_1 \Sigma_g\to E_\Gamma \to \Gamma\to 1.\]
We refer the reader to the reference \cite{FMbook}   for related facts about $\mathbb{Mod}(\Sigma_g)$ and $\mathcal T(\Sigma_g)$.

In \cite{FarbMosher},  Farb and Mosher studied when the extension is a hyperbolic group and showed that, when $\Gamma$ is a Schottky group,  this is equivalent to the quasiconvexity of $\Gamma$-orbits in $\mathcal T(\Sigma_g)$.

In Theorem \ref{thm:gen2}, the quasi-isometrically embedded image of the free group $\Gamma$ are contracting and thus quasiconvex in the sense of Farb and Mosher. Thus, by \cite[Theorem 1.1]{FarbMosher}, the free group $\Gamma$ is convex-cocompact in their sense, so the following result holds.

\begin{Theorem}
The set of $k$-tuples of elements $(u_1, u_2, \cdots, u_k)$ in $\mathbb{Mod}(\Sigma_g)$ with the hyperbolic extension in $\mathbb{Mod}(\Sigma_g, p)$ is exponentially generic.
\end{Theorem}

Our second  main result obtains the statistical hyperbolicity for the exact class of actions  as in Theorem \ref{thm:gen}, and in particular for statistically convex-cocompact actions.

\begin{Theorem}\label{thm:stahyper}
Let a non-elementary group $G$ act properly on $(Y,d)$ with a contracting element satisfying DOP condition and  purely exponentially growth. Then $G$ is statistically hyperbolic in balls and annuli. In particular, if the action is SCC, then $G$ is statistically hyperbolic in balls and annuli
\end{Theorem}

\begin{Remark}
Motivated by the distinction between SCC action and a general proper action, one may wonder whether there is a significant convergence rate of $E_A(G, \Delta)$ or $E_B(G)$ under  SCC actions. This is, however, not true  even in free groups: a simple computation as Example \ref{FreeExample} shows that the convergence rate is of order $\frac{1}{n}$. Hence,  we have no assertion on  the convergence speed.
\end{Remark}

Except the class of SCC actions,  the action of  discrete groups on CAT(-1) spaces provides a source of examples with DOP condition and purely exponential growth. For example, combining   \cite{Roblin} and \cite{PS18}  we obtain that the finiteness of  the Bowen-Margulis-Sullivan measure on the geodesic flow is equivalent to either have purely exponential growth or satisfy the DOP condition. Hence, we obtain the following corollary.

\begin{Theorem}
Suppose that the Bowen-Margulis-Sullivan measure on the unit tangent bundle of a Hadamard manifold is finite. Then the fundamental group action on the universal covering is statistically hyperbolic in balls and annuli. Moreover, the generic pair of elements generate a free group of rank 2 with uniform quasi-isometric embedding.
\end{Theorem}

If a hyperbolic $n$-manifold for $n\ge 2$ is geometrically finite, then the BMS measure is always finite \cite{Sul}.  We thus have the following corollary in Kleinian groups, which seems to be not recorded in literatures. Note that examples of non-geometrically finite Kleinian groups with finite BMS measures are constructed for $n\ge 4$ by Peign\'e in \cite{Peigne}.

\begin{Corollary}
Geometrically finite Kleinian groups are statistically hyperbolic and have generic free basis property.
\end{Corollary}

For  the action of  mapping class groups on  Teichm\"{u}ller space, we then have the following corollary, which could be thought of as a discrete analogue of the result in \cite{DDM}.

\begin{Corollary}
The action of mapping class groups on  Teichm\"{u}ller space is statistically hyperbolic with respect to the Teichm\"{u}ller  metric.
\end{Corollary}

Of course, the action of a group on the Cayley graph is SCC, so if there exists a contracting element, then it is statistically hyperbolic. This allows us to give   new examples of groups with statistically hyperbolic property in the original sense \cite{DLM}.

\begin{Corollary}
The following classes of groups are statistically hyperbolic with respect to word metrics.
\begin{enumerate}

\item a Gr'$(\frac{1}{6})$-labeled graphical small cancellation group $G$ with finite components labeled by a finite set $S$ acts on the Cayley graph $\mathscr{G}(G,S)$ with respect to the finite generating set $S$.
\item
Right-angled Artin (Coxeter) groups are statistically hyperbolic with respect to the standard generating set, if they are not virtually direct product.
\end{enumerate}

\end{Corollary}
We point out that it is not clear to us whether the above two classes of groups   are    statistically hyperbolic   for every   generating set.  Note that they  include non-relatively hyperbolic examples of groups (cf. \cite{BDM, GruS}). Hence,  it would be interesting to know to which extent  the statistical hyperbolicity for every generating set characterizes the class of relatively hyperbolic groups.

\vspace{.5em}
\hspace{-1.5em}\textbf{The structure of this paper}\,\,Section \ref{Sect2} discusses the notions  and relevant facts of contracting elements, SCC actions and the DOP condition. The main technical contribution is given in Section \ref{sec:negligiblesubset} and provides   useful characteristics of several negligible sets. In Section \ref{SecProofs}, a generic set of elements is then singled out to complete the proofs of Main Theorems.

\section{ Preliminaries}\label{Sect2}

In this section, we will introduce some preliminaries. First we fix some notations and conventions.

\subsection{Notations and Conventions}

Let $(Y,d)$ be a proper geodesic metric space. The $r$ neighborhood of a subset $X\subseteq Y$ is   denoted by $N_r(X)$. We denote $\|X\|$ by the diameter of a subset $X\subseteq Y$ and $d_{Haus}(X_1,X_2)$ by the Hausdorff distance of two subsets $X_1,X_2\subseteq Y$. Given a point $y\in Y$, and a subset $X\subseteq Y$, let $\Pi_X(y)$ be the set of point $x$ in $X$ such that $d(y,x)=d(y,X)$. The \textit{projection} of a subset $A\subseteq Y$ to $X$ is then $\Pi_X(A):=\cup_{a\in A}\Pi_X(a)$.

The path $\gamma$  in $Y$ under consideration is always assumed to be rectifiable with arc-length parametrization $[0, |\gamma|]\to \gamma$, where $|\gamma|$ denotes the length of $\gamma$. Denote by $\gamma_-,\gamma_+$ the initial and terminal points of $\gamma$ respectively. For any two parameters $a<b\in [0, |\gamma|]$,  we   denote by $[\gamma(a),\gamma(b)]_\gamma:=\gamma([a,b])$ and $(\gamma(a),\gamma(b))_\gamma:=\gamma((a,b))$ the closed (resp. open) subpath of $\gamma$ between $a$ and $b$. For any $x,y\in Y$, we denote by $[x,y]$ a choice of geodesic in $Y$ from $x$ to $y$.

Given a property (P), a point $z$ on $\gamma$ is called the \textit{entry point} satisfying (P) if $|[\gamma_-,z]_\gamma|$ is minimal among the points $z$ on $\gamma$ with the property (P). A point $w$ on $\gamma$ is called the \textit{exit point} satisfying (P) if $|[w,\gamma_+]_\gamma|$ is minimal among the points $w$ on $\gamma$ with the property (P).

A path $\gamma$ is called a $(\lambda,c)$-quasi-geodesic for $\lambda\geq1,c\geq0$ if the following holds
\[|{\beta}|\leq\lambda\cdot d(\beta_-,\beta_+)+c\]
for any rectifiable subpath $\beta$ of $\alpha$.

Let $\beta,\gamma$ be two paths in $Y$. Denote by $\beta\cdot\gamma$ (or simply $\beta\gamma$) the concatenated path provided that $\beta_-=\gamma_+$.

Let $f,g$ be real-valued functions with domain understood in the context. Then $f\prec_{c_i}g$ means that there is a constant $a>0$ depending on parameters $c_i$ such that $f<ag$. The symbols $\succ_{c_i}$ and $\asymp_{c_i}$ are defined analogously. For simplicity, we shall omit $c_i$ if they are universal constants.

We say a sequence $\{a_n\}\subseteq \mathbb{R}$ of numbers converges to a number $\lambda\in\mathbb{R}$ \textit{exponentially fast}, denoted by $a_n \stackrel{exp}{\rightarrow} \lambda$,
if $$|\lambda-a_n|\leq c\theta^n$$ for some constant $\theta\in(0,1)$ and a positive constant $c>0$.

\begin{Remark}
\begin{enumerate}
\item It is clear that the (exponential) genericity is preserved by taking any finite intersection and finite union.   This fact shall be often used implicitly.
\item If $X$ is exponentially negligible, then $\delta_X<\delta_G$, by which we call $X$ \textit{growth tight} in \cite{Y1}. Note that if $G$ has purely exponentially growth, then a growth tight set is exponentially negligible. In this paper, the group actions under consideration always have purely exponentially growth, so we do not distinguish these two notions.
\end{enumerate}
\end{Remark}

\subsection{Contracting Property}

We fix a preferred class of quasi-geodesics $\mathcal{L}$, which contains at least all geodesics in $Y$.

\begin{Definition}[Contracting subset]
A subset $X\subseteq Y$ is called $\kappa$-\textit{contracting} with respect to $\mathcal{L}$ if for any quasi-geodesic $\gamma\in\mathcal{L}$ with $d(\gamma,X)\geq \kappa$, we have $\Pi_X(\gamma)\leq \kappa$.
A collection of $\kappa$-contracting subsets is referred to as a $\kappa$-\textit{contracting system} (with respect to $\mathcal{L}$).
\end{Definition}

We first note the following examples in various contexts.
\begin{Examples}
\begin{enumerate}
\item Quasi-geodesics and quasi-convex subsets are contracting with respect to the set of all
quasi-geodesics in hyperbolic spaces.
\item Fully quasi-convex subgroups (and in particular, maximal parabolic subgroups) are contracting with respect to the set of all quasi-geodesics in relatively hyperbolic groups(see \cite[Proposition 8.2.4]{GP}).
\item The subgroup generated by a hyperbolic element is contracting with respect to the set of all quasi-geodesics in groups with non-trivial Floyd boundary (see \cite[Section 7]{Y1}).
\item  Contracting segments in CAT(0)-spaces in the sense of Bestvina and Fujiwara are contracting here with respect to the set of geodesics (see \cite[Corollary 3.4]{BF}).
\item The axis of any pseudo-Anosov element is contracting relative to geodesics in Teichm\"{u}ller  spaces by Minsky \cite{M}.
\end{enumerate}
\end{Examples}

\begin{Convention}
In view of the above examples, the preferred collection $\mathcal{L}$ in the sequel will always be the set of all geodesics in $Y$.
\end{Convention}

The notion of a contracting subset is equivalent to the following one considered by Minsky \cite{M}. The proof given in \cite[Corollary 3.4]{BF} for CAT(0) spaces is valid in the general case. In this paper, we will always work with the above definition of the contracting property.

\begin{Lemma}
A subset $X$ is contracting in $Y$ if and only if any open ball $B$ missing $X$ has a uniformly bounded projection to $X$.
\end{Lemma}

We collect some properties of contracting sets that will be used later on. The proof is straightforward and is left to the interested reader.

\begin{Lemma}\label{lemma:prop}
Let $X$ be a contracting set.
\begin{enumerate}
\item (Quasi-convexity) $X$ is $\sigma$-quasi-convex for a function $\sigma:\mathbb{R}_{\geq0}\rightarrow\mathbb{R}_+$: given $r\geq0$, any geodesic with endpoints in $N_r(X)$ lies in the neighborhood $N_{\sigma(r)}(X)$.
\item (Finite neighborhood) Let $Z$ be a set with finite Hausdorff distance to $X$. Then $Z$ is contracting.
\item there exists a constant $C>0$ such that  for any geodesic segment $\gamma$,
\begin{equation} \label{eq:pointproj}
\big|\|\Pi_X(\{\gamma_-,\gamma_+\})\|-\|\Pi_X(\gamma)\|\big|\leq {C}.
\end{equation}
\end{enumerate}
\end{Lemma}

In most situation, we are interested in a contracting system $\mathbb{X}$ with a $\nu$-\textit{bounded intersection} for a function $\nu:\mathbb{R}_{\geq0}\rightarrow\mathbb{R}_{\geq0}$ if the following holds
\[\forall X\neq X'\in\mathbb{X},\,\|N_r(X)\cap N_r(X')\|\leq\nu(r)\]
for any $r\geq0$. This property is, in fact, equivalent to a bounded projection property of $X$: there exists a constant $B>0$ such that $$\Pi_X(X')\leq B$$ for $X\neq X'\in\mathbb{X}$. See \cite{Y1} for further discussions.

An infinite subgroup $H<G$ is called \textit{contracting} if for some (hence any by \cite[Proposition 2.4.2]{Y}) $o\in Y$, the subset $Ho$ is contracting in $Y$.

An element $h\in G$ is called \textit{contracting} if the subset $\langle h\rangle o$ is contracting, and the orbital map $n\in\mathbb{Z}\mapsto h^no\in Y$ is a quasi-isometric embedding. The set of contracting elements is preserved under conjugacy.

Let $H$ be a contracting subgroup. We define a group $E(H)$ as follows:
\[E(H):=\{g\in G:\exists r>0,gHo\subseteq N_r(Ho),Ho\subseteq N_r(gHo)\}.\]

For a contracting element $h$, we have the following result about $E(h):=E(\langle h\rangle)$  (see \cite[Lemma 2.11]{Y}).

\begin{Lemma}\label{lemma:finiteindex}
Assume that $G$ acts properly on $(Y,d)$. For a contracting element $h$, the following statements hold:
\begin{enumerate}
\item $E(h)=\{g\in G:\exists n>0,gh^ng^{-1}=h^n\,or\,gh^ng^{-1}=h^{-n}\}$.
\item $[E(h):\langle h\rangle]<\infty$, and $E(h)$ is a contracting subgroup with bounded intersection.
\end{enumerate}
\end{Lemma}

The contracting subset $\ax(h):=\{f\cdot o:f\in E(h)\}$ shall be referred to as the \textit{axis} of $h$. In the following  discussion, so we denote $A=\ax(h)$ for simplicity.

\begin{Lemma}\label{lemma:proj}
For any $C>0$, let $\gamma$ be a geodesic with interior dose not meet $N_{{C}}(A)$. Then
\[d_{Haus}(\Pi_{N_{{C}}(A)}(\gamma),\Pi_{A}(\gamma))\leq  {C}.\]
In particular, if $ {C}$ is a contracting constant of $A$, then we have $\Pi_{N_{ {C}}(A)}(\gamma)\leq 3 {C}$.
\end{Lemma}

\begin{proof}
	For any $x\notin N_{C}(A)$, it is sufficient to prove \[d_{Haus}(\Pi_{N_{C}(A)}(x),\Pi_{A}(x))\leq C.\]
	
	For any $y\in\Pi_{A}(x)$, take some $[x,y]$. Let $z$ be the point of $[x,y]$ such that $d(y,z)=C$. Now for each $z'\in N_{C}(A)$, there exists $y'\in A$ such that $d(y',z')\leq C$. Since
	\[d(x,z')+d(z',y')\geq d(x,y')\geq d(x,y)=d(x,z)+d(z,y),\]
	we have $d(x,z')\geq d(x,z)$, which means that $z\in \Pi_{N_{C}(A)}(x)$. Thus
	\[\Pi_{A}(x)\subseteq N_{C}(\Pi_{N_{C}(A)}(x)).\]
	
	Let $z\in \Pi_{N_{C}(A)}(x)$. Since $z\in N_{C}(A)$, there exists $y\in A$ so that $d(y,z)\leq C$.
	Take some $y'\in \Pi_{A}(x)$, then there exists $z'\in \Pi_{N_{C}(A)}(x)$ so that $d(x,y')=d(x,z')+C$ by the above discussion. Then
	\[d(x,y)\leq d(x,z)+d(y,z)\leq d(x,z)+C=d(x,z')+C=d(x,y').\]
	So $y\in \Pi_{A}(x)$, $\Pi_{N_{C}(A)}(x)\subseteq N_{C}(A)$.
\end{proof}

\begin{Lemma}\label{lemma:notint}
Let $C>0$ be the contraction constant of $A$  and $\alpha,\beta$ be two geodesics with the same initial endpoint. If $x$ is the entry point of $\alpha$ into $N_{ {C}}(A)$ and  $\beta\cap B(x,4 {C})=\emptyset$, then $\beta\cap N_{ {C}}(A)=\emptyset$.
\end{Lemma}

\begin{proof}
If $\beta\cap N_{ {C}}(A)\neq\emptyset$, then let $y\in\beta$ be the entry point of $\beta$ in $N_{ {C}}(A)\neq\emptyset$. We have
\[d(x,y)\leq  {C}+\|\Pi_{A}([\alpha_-,x]_{\alpha})\|+\|\Pi_{A}([\beta_-,y]_{\beta})\|+ {C}\leq4 {C}\]
which proves the lemma.
\end{proof}

 Since $gN_{ {C}}(A)=N_{ {C}}(gA)$ for every $g\in G$,  the following lemma is a consequence of Lemma \ref{lemma:finiteindex}, Lemma \ref{lemma:proj} and Lemma \ref{lemma:prop}.

\begin{Lemma}\label{lemma:strongcontracting}
For any ${C}\geq 0$,  the collection $\mathbb{X}=\{gN_{{C}}(A): g\in G\}$ is a contracting system with bounded projection.
\end{Lemma}

\subsection{Admissible Path}
Let $\mathbb X$ be a contracting system with a bounded intersection property. The following notion of an admissible path will be used to obtain   a quasi-geodesic path.
\begin{Definition}[Admissible Path]Given $D,\tau\geq0$ and a function $\mathcal{R}:\mathbb{R}_{\geq0}\rightarrow\mathbb{R}_{\geq0}$, a path $\gamma$ is called $(D,\tau)$-\textit{admissible} in $Y$, if $\gamma$ is a concatenation of geodesic subpaths $p_0q_1p_1\cdots q_np_n$ $(n\in\mathbb{N})$($p_0,p_n$ could be trivial), the endpoints of $p_i$ are in some $X_i\in \mathbb{X}$ for each $i$, and satisfies the following called \textit{Long Local} and \textit{Bounded Projection} properties:
\begin{enumerate}
\item[(LL1)] Each $p_i$ has length bigger than $D$, except that $(p_i)_-=\gamma_-$ or $(p_i)_+=\gamma_+$;
\item[(BP)] For each $X_i$, we have $\max\{\|\Pi_{X_i}(q_i)\|,\|\Pi_{X_i}(q_{i+1})\|\}\leq\tau$, where $q_0:=\gamma_-$ and $q_{n+1}:=\gamma_+$ by convention;
\item[(LL2)]Either $X_i\neq X_{i+1}$ have $\mathcal{R}$-bounded intersection or $q_{i+1}$ has length bigger than $D$.
\end{enumerate}
\end{Definition}

We need the following result from \cite[Corollary 3.2]{Y1}.

\begin{Proposition}\label{pro:quasi-geo}
Let $\kappa$ be the contraction constant of $\mathbb{X}$. For any $\tau>0$, there are constants $D_0=D_0(\kappa,\tau)>0,\Lambda=\Lambda(\kappa,\tau)>0$ such that given $D>D_0$ any $(D,\tau)$-admissible path is $(\Lambda,0)$-quasi-geodesic.
\end{Proposition}

We refer the reader to  \cite{Y1,Y} for   further discussions about admissible path.

\subsection{SCC actions and barrier-free elements} \label{sec:barrierfree}

We recall the notion of a barrier-free element from \cite{Y}.

\begin{Definition}
Fix constants $\nu,M>0$.
\begin{enumerate}
\item
Given $\nu>0$ and $g\in G$, we say that a geodesic $\gamma$ contains an $(\nu,g)$-\textit{barrier} if there exists a element $z\in G$ so that
    \begin{equation} \label{eq:barrier}
    \max\{d(z\cdot o,\gamma),d(z\cdot go,\gamma)\}\leq\nu.
    \end{equation}
    If no such $z\in G$ exists so that (\ref{eq:barrier}) holds, then $\gamma$ is called $(\nu,g)$-\textit{barrier-free}.

\item
An element $f\in G$ is $(\nu,M,g)$-\textit{barrier-free} if there exists an $(\nu,g)$-barrier-free geodesic between $B(o,M)$ and $B(fo,M)$.
\end{enumerate}
\end{Definition}

We have chosen two parameters $M_1,M_2$ so that the definition of a statistically convex-cocompact action \ref{SCCDefn} is flexible and easy to verify. It is enough to take $M_1=M_2=M$ in our use. Henceforth, we set $\mathcal{O}_M:=\mathcal{O}_{M,M}$ for easy of notation.
 When the SCC action contains a  contracting element, the definition is independent of the basepoint (see \cite{Y}).

Given $\nu,M>0$ and  any $g\in G$, let $\mathcal{V}_{\nu,M,g}$ be the collection of all $(\nu,M,g)$-barrier-free elements of $G$. The following results will be key in next sections.

\begin{Proposition}\label{pro:purelyexp}\cite{Y}
If $G$ admit a SCC action on a proper geodesic space $(Y,d)$ with a contracting element, then
\begin{enumerate}
\item $G$ has purely exponentially growth.
\item Let $M_0$ be the constant in the definition of SCC action, then for any $M>M_0$, there exists $\nu=\nu(M)>0$ such that $\mathcal{V}_{\nu,M,g}$ is exponentially negligible for any $g\in G$.
\end{enumerate}
\end{Proposition}

It is easy to see from the proof of  \cite[Corollary 4.5]{Y} that the following conclusion holds in a general proper action.

\begin{Proposition}\label{pro:neg}
Suppose that a group $G$ acts properly on a proper geodesic space $(Y,d)$ with a contracting element, then
for any $M>0$, there exists $\nu=\nu(M)>0$ so that
\[\sum_{n=1}^{+\infty}|\mathcal{V}_{\nu,M,g}\cap A(n,\Delta)|\exp(-n\delta_G)<+\infty\]
for any $g\in G$.
\end{Proposition}

\subsection{The DOP condition}
This subsection collects several useful consequences of the Dal'bo-Otal-Peign\'{e} condition.
For any $0 \le n_1\le n_2$, we consider the following annulus-like set $$A([n_1,n_2],\Delta):=\{g\in G:n_1-\Delta\leq d(o,go)\leq n_2+\Delta\}.$$ Usually, we consider the \textit{$(\rho, \Delta)$-annulus} $A([\rho n,n], \Delta)$ for $\rho\in [0, 1]$.  For simplicity, we write $A([\rho n,n])$ if $\Delta=0$, and assume that $\rho n$ are integers.

Observe that
    \begin{equation} \label{eq:DOP}
    \sum_{g\in \mathcal{O}_{M_1,M_2}}|g|\exp(-\delta_G|g|)\asymp_\Delta \sum_{n=1}^{+\infty}n|\mathcal{O}_{M_1,M_2}\cap A(n,\Delta)|\exp(-n\delta_G),
    \end{equation}
for any $\Delta>0$. Indeed, this follows from the fact that any $g \in \mathcal{O}_{M_1,M_2}$ is contained in a uniform number of annular sets $A(n,\Delta)$ where $n\ge 1$. Consequently,
    \begin{equation} \label{eq:converge}
    \sum_{g\in \mathcal{O}_{M_1,M_2}}\exp(-\delta_G|g|)<\infty.
    \end{equation}

Thus, if $G$ admit a SCC action on $Y$, then the action satisfies the DOP condition.  We remark that the formula (\ref{eq:converge})  turns out to be  true for any proper action of $G$ on $(Y,d)$ with a contracting element: the methods in \cite{Y} can be invoked to prove (\ref{eq:converge}). This generality is not used here and so the details are left to interested reader.

For any $\Delta>0$, let
\[\mathcal{O}_M(n,\Delta):=\mathcal{O}_M\cap A(n,\Delta)\cup\{1\},~
\mathcal{V}_{\nu, h}(n,\Delta):=\mathcal{V}_{\nu, h}\cap A(n,\Delta).\]

The following elementary lemma will be needed in the next section.

\begin{Lemma}\label{lemma:0}
Assume that the proper group action satisfies the DOP condition. For any $1>\varepsilon>0$ and any   $\Delta>0$, we have
\begin{enumerate}
\item \[\lim_{n\rightarrow\infty}\sum_{\varepsilon n\leq l\leq n}n|\mathcal{O}_M(l,\Delta)|\exp(-l\delta_G)=0.\]
\item \[\lim_{n\rightarrow\infty}\sum_{\epsilon n \le l \le n}\sum^{l_1+l_2+l_3=l}_{l_1, l_2, l_3\geq 0} (l_1+1) |\mathcal{O}_M(l_1,\Delta)|\cdot|\mathcal{V}_{\nu, h}(l_2,\Delta)|\cdot |\mathcal{O}_M(l_3,\Delta)| \cdot \exp(-n\delta_G)=0.\]
When the action is SCC, the convergence is exponentially fast.
\end{enumerate}
\end{Lemma}

\begin{proof}
By  definition of the DOP condition, we obtain
\[\sum_{n=0}^{+\infty}|\mathcal{O}_M(n,\Delta)|\exp(-n\delta_G)<\infty.\]
from the formulae (\ref{eq:DOP}) and (\ref{eq:converge}). By the Cauchy criterion of series, we know
\[\lim_{n\rightarrow\infty}\sum_{\varepsilon n\leq l\leq n}l|\mathcal{O}_M(l,\Delta)|\exp(-l\delta_G)=0\]
where the convergence is exponential fast when the action is SCC.
The first statement (1) thus follows from the following
\[\sum_{\varepsilon n\leq l\leq n}\varepsilon n|\mathcal{O}_M(l,\Delta)|\exp(-l\delta_G)
\leq \sum_{\varepsilon n\leq l\leq n}l|\mathcal{O}_M(l,\Delta)|\exp(-l\delta_G).\]

By Proposition \ref{pro:neg}, we have $\displaystyle\sum_{n=1}^{+\infty}|\mathcal{V}_{h^m}(n,\Delta)|\exp(-n\delta_G)<\infty$, where the partial sum converges exponentially fast when the action is SCC. The second statement then follows from the convergence of the Cauchy product of three convergent series.   The proof is finished.
\end{proof}

At last, we introduce a slightly general notion of negligibility using $(\rho, \Delta)$-annulus. Fix a   number $\rho \in (0, 1]$ and $\Delta>0$. We say that a set $K\subset G$ is \textit{negligible} in the $(\rho, \Delta)$-annulus  if the following holds
\begin{equation}\label{bigAnnuliConv}
\displaystyle \frac{|K\cap A([\rho n,n], \Delta)|}{|A([\rho n,n], \Delta)|} \to 0.
\end{equation}
If the convergence is exponentially fast, it is \textit{exponentially negligible}.

The following lemma clarifies its role in proving the genericity in the next sections.   It follows immediately from the purely exponential growth.
\begin{Lemma}\label{LargeAnnulus}
Assume that the proper group action has purely exponential growth. For any $0<\rho<1$, we have $|A([\rho n,n])|\asymp_\rho \exp(\delta_Gn)$ and
\[\frac{|A([\rho n,n])\times A([\rho n,n])|}{|B_{n}\times B_{n}|}\stackrel{exp}{\rightarrow}1,\frac{|A([\rho n,n])|}{|B_{n}|}\stackrel{exp}{\rightarrow}1.\]
\end{Lemma}

Hence, in order to prove that a set $K$ is (exponentially) negligible in $G$, we can assume that $K\subset A([\rho n,n])$ to simplify the discussion for a certain choice of $\rho\in(0,1)$. That is to say, we only need to prove that $K$ is (exponentially) negligible  in $(\rho, \Delta)$-annulus. And, it turns out that the proof of (\ref{bigAnnuliConv}) for $\rho=1$ is much more simple than that for $\rho\in(0,1)$. Therefore, we shall consider the big annulus instead of the usual one in next sections.

The same consideration applies in the case of $G^{(2)}$ where $K$ is assumed to be in $A([\rho n,n])\times A([\rho n,n])$.

\begin{comment}
The reader may be wondering about the necessity of introducing the complicated   set $A([\rho n,n])$ instead of the usual one $A(n,\Delta)$.  This concern is particularly underlined by the following fact from \cite[Lemma 3.1]{GMO}: if for a given subset $K\subseteq G$, we have
\begin{equation}\label{AnnuliConv}
\frac{|K\cap A(n,\Delta)|}{|A(n,\Delta)|}{\rightarrow}0
\end{equation}
exponentially fast for some $\Delta>0$, then $K$ is exponentially negligible in $G$ using ball model. We do have this exponential convergence when the group action is SCC.  However, for a general proper action as assumed in Theorems \ref{thm:gen} \ref{thm:stahyper},   the convergence  of (\ref{AnnuliConv})  is guaranteed by  Proposition \ref{pro:neg} but without the exponential rate. As a consequence, we are unable to obtain the genericity in ball model only using (\ref{AnnuliConv}). In addition, the proof of  (\ref{AnnuliConv}) is much more simple than that of (\ref{bigAnnuliConv}).

Therefore, for this reason, we shall consider the big annulus instead of the usual one in next sections, and introduce the following technical notions to unify them.
\end{comment}

\section{Negligible subsets} \label{sec:negligiblesubset}

Throughout this section, let $G$ admit a proper action on a proper geodesic metric space $(Y,d)$ with a contracting element. If the group action   satisfies the DOP condition, then we take $\nu,M>0$ to satisfy the definition of DOP condition and Proposition \ref{pro:neg}.  When the action is SCC, the constants $\nu,M>0$ are given by Proposition \ref{pro:purelyexp}.   We denote $\mathcal{O}_M=\mathcal{O}_{M,M},\mathcal{V}_{\nu, h}=\mathcal{V}_{\nu,M,h}$ for simplicity.

The goal of this section is to provide some negligible sets under the above assumptions. Moreover, these are exponentially negligible when the group action is SCC. We suggest that the reader only reads the definition of these sets first and then read the proof of the theorems in next section, finally return to the proof that these sets are negligible.

In all results obtained in what follows, we assume  in the DOP case and have in the SCC case by Proposition \ref{pro:purelyexp} that $G$ has purely exponentially growth:
$$|B_n|\asymp \exp(\delta_Gn) \asymp_\Delta |A(n,\Delta)|$$
for any $\Delta\gg 0$. We fix such a constant $\Delta$. This estimate will be used implicitly several times.

\subsection{Elements with definite barrier-free proportion}
This subsection defines three negligible subsets of elements with definite proportion with(out) certain properties.

For any $\varepsilon\in (0, 1)$, let ~$U(\varepsilon)$~ be the set of elements $u\in G$ such that some geodesic $\alpha=[o,uo]$ contains a subsegment $\alpha^{\varepsilon}$ of length $\varepsilon|u|$   outside $N_M(Go)$. That is to say,
\begin{equation}\label{Uepsilon}
U(\varepsilon)=\{u\in G: \exists\alpha=[o,uo], \alpha^{\varepsilon} \subset \alpha~ s.t.~|\alpha^\varepsilon|\ge \epsilon |\alpha|,~\alpha^\varepsilon \cap N_M(Go)=\emptyset.\}
\end{equation}

\begin{Lemma}\label{lemma:out}
If the   action has PEG and satisfies the DOP condition, then for any $\varepsilon\in (0, 1)$ and $1\ge \rho>\varepsilon$, we have $U(\varepsilon)$ is negligible in $(\rho, \Delta)$-annuli.

Moreover, if the action is SCC, then $U(\varepsilon)$ is exponentially negligible.
\end{Lemma}

\begin{proof}
Assume first that the group action satisfies the PEG and DOP condition.

Fix any $1>\rho>\varepsilon$. By Lemma \ref{LargeAnnulus}, we only need to show that  $\displaystyle\frac{|U(\varepsilon)\cap A([\rho n,n])|}{|A([\rho n,n])|}\rightarrow 0$ as $n\rightarrow\infty$. Consider any $g\in U(\varepsilon)\cap A([\rho n,n])$ and denote $|g|=k$, then $\rho n\leq k\leq n$. By definition of $U(\varepsilon)$, there exists  a geodesic $\alpha=[o,go]$ such that
\begin{equation} \label{eq:out}
\alpha ~\text{contains a subsegment of length}~ \varepsilon k ~\text{which lies outside}~ N_M(Go).
\end{equation}
Among those, we consider the first maximal open
segment $(x,y)_\alpha$ of $\alpha$ which lies outside $N_M(Go)$ and whose length is bigger than $\varepsilon k\in[\varepsilon\rho n, n]$.

According to the length  and the position of $(x, y)_\alpha$,  we subdivide $U(\varepsilon)\cap A([\rho n,n])$ into a sequence of subsets as follows.

For $0\leq i\leq (1-\varepsilon)n,\varepsilon\rho n\leq l\leq n$, define $U_i^l$ to be the set of element $g\in U(\varepsilon)\cap A([\rho n,n])$ such that the segment $(x,y)_\alpha\subseteq\alpha$ defined as above satisfies $d(o, x)=i$ and $d(x, y)=l$.
Then we have the following decomposition,
\[U(\varepsilon)\cap A([\rho n,n])=\bigcup_{\substack{0\leq i\leq (1-\varepsilon)n \\ \varepsilon\rho n\leq l\leq n}}U_i^l.\]

For any $g\in U_i^l$, there exists a geodesic $\alpha=[o,go]$ such that $\alpha_{(i,i+l)}$ lies outside $N_M(Go)$ and $\max\{d(\alpha(i),uo),d(\alpha(i+l),vo)\}\leq M$ for some $u,v\in G$. Now we can write  $g=u(u^{-1}v)(v^{-1}g)$, where
\[u\in A(i,M),u^{-1}v\in \mathcal{O}_M(l,2M),v^{-1}g\in A([\rho n-l-i,n-l-i],M)\subseteq B_{n-l-i+M}.\]

Set $\Delta=2M$. We assumed that $G$ has purely exponential growth, so $$|A(n, \Delta)|\asymp_\Delta \exp(\delta_G n)\asymp |B_n|.$$ We thus obtain
\begin{align*}
|U(\varepsilon)\cap A([\rho n,n])|
&\leq \sum_{\substack{0\leq i\leq (1-\varepsilon)n \\ \varepsilon\rho n\leq l\leq n}} |U_i^l| \\
&\leq \sum_{\substack{0\leq i\leq (1-\varepsilon)n \\ \varepsilon\rho n\leq l\leq n}}
|A(i,\Delta)|\cdot|\mathcal{O}(l,\Delta)|\cdot|B_{n-l-i+\Delta}| \\
&\prec_\Delta \sum_{\substack{0\leq i\leq (1-\varepsilon)n \\ \varepsilon\rho n\leq l\leq n}}
 \exp(i\delta_G)\cdot|\mathcal{O}_M(l,\Delta)|\cdot  \exp((n-l-i)\delta_G) \\
&\prec_\Delta  \sum_{\varepsilon\rho n\leq l\leq n}n|\mathcal{O}_M(l,\Delta)|\exp((n-l)\delta_G).
\end{align*}
Therefore, the negligibility of $U(\varepsilon)$ follows from Lemma \ref{lemma:0}.

If the group action is SCC, then there exists $0<\delta_{\mathcal O}<\delta_G$ such that $|\mathcal{O}_M(l,\Delta)|\prec_\Delta \exp(l \delta_{\mathcal O})$. The above computation goes without changes, and so we get
\begin{align*}
|U(\varepsilon)\cap A([\rho n,n])|
&\prec_\Delta  \sum_{\varepsilon\rho n\leq l\leq n}n|\mathcal{O}_M(l,\Delta)|\exp((n-l)\delta_G) \\
&\prec_\Delta   \sum_{\varepsilon\rho n\leq l\leq n}n \exp(l\delta_{\mathcal{O}_M})\cdot\exp((n-l)\delta_G) \\
&\prec_\Delta   n^2 \exp(-(\delta_G-\delta_{\mathcal{O}})\varepsilon\rho n)\exp(n\delta_G).
\end{align*}
Hence, in this case, $U(\varepsilon)$  is exponentially negligible.
\end{proof}

Let $h\in G$ be a contracting element with the axis $\ax(h)=E(h)\cdot o$, where $E(h)$ is the maximal elementary subgroup given in Lemma \ref{lemma:finiteindex}.

Given $\varepsilon\in (0,1)$ and $C>0$,  consider the following set of elements $g\in G$ such that an $\epsilon$-percentage   of $[o, go]$  is contained in some translate of $\ax(h)$.
\begin{equation}\label{Wepsilon}
W(\varepsilon, h, C)=\{g\in G:\exists\alpha=[o,go], \alpha^{\varepsilon} \subset \alpha \cap  N_C(f \ax(h))~ s.t.~|\alpha^\varepsilon|\ge \epsilon |\alpha|~\text{for some}~ f\in G\}.
\end{equation}

\begin{Lemma}\label{lemma:short}
Assume that the   action has PEG. For any $0<\varepsilon<\rho\leq1$ and $C>0$, we have $W(\varepsilon, h, C)$ is exponentially negligible in $(\rho, \Delta)$-annuli in $G$.
\end{Lemma}

\begin{proof}
Since $h\in G$ is contracting, and by definition, $i\mapsto h^io$ is a quasi-isometric embedding, we have $|\langle h \rangle \cap B_n|\asymp n$. By Lemma \ref{lemma:finiteindex},   we have $[E(h):\langle h\rangle] <\infty$, so the following holds  $$|  E(h)   \cap B_n|\asymp n.$$

As before, we want to show $\displaystyle\lim_{n\rightarrow\infty}\frac{|W(\varepsilon, h, C)\cap A([\rho n,n])|}{|A([\rho n,n])|}=0$. Let $g\in W(\varepsilon)\cap A([\rho n,n])$, so $\rho n\leq j\leq n$, where $j:=|u|$. By definition of $W(\varepsilon,  h, C)$, there exists $\alpha=[o,go]$, $i\in[0,(1-\varepsilon)j]$ and $f\in G, k\in E(h)$ such that
\[ d(\alpha(i), fo)\leq C,\;d(\alpha(i+\varepsilon j), fko)\leq C.\]
Thus, we have $f\in A(i,C)$ and $d(o,ko)\leq\varepsilon j+2C\leq \varepsilon n+2C$, which yields that $k\in E(h)\cap B_{\varepsilon n+2C}$. Consequently, we can write $g=f k((f k)^{-1}g)$ where $(f k)^{-1}g\in B_{n-i-\varepsilon\rho n+C}$. This gives the following
\[W(\varepsilon,  h, C)\cap A([\rho n,n])\subseteq\bigcup_{i=0}^{(1-\varepsilon)n}A(i,C)\cdot(E(h)\cap B_{\varepsilon n+2C})\cdot B_{n-i-\varepsilon\rho n+C}.\]

Since $G$ has purely exponentially growth, we have the following estimate:
\begin{align*}
|W(\varepsilon,  h, C)\cap A([\rho n,n])|&\leq \sum_{i=0}^{(1-\varepsilon)n}|A(i,C)|\cdot| E(h)\cap B_{\varepsilon n+2C}|\cdot|B_{n-i-\varepsilon\rho n+C}|\\
&\prec n\cdot n\cdot \exp((1-\varepsilon\rho) n\delta_G)
\end{align*}

which clearly concludes the proof of the result.
\end{proof}

We now introduce the third negligible sets of elements which have a fixed percentage being barrier-free. To be precise, we need a bit more notation. Let $\alpha$ be a geodesic and $\varepsilon_1 \le \varepsilon_2\in [0, 1]$. We denote by $\alpha_{[\epsilon_1, \epsilon_2]}$ the subsegment $\alpha([\varepsilon_1 n,\varepsilon_2 n])$ of $\alpha$, where $n=|\alpha|$.

Given $0<\varepsilon_1<\varepsilon_2<1$ and $h\in G$, {we define}
\begin{equation}\label{Vepsilon12}
V(\varepsilon_1,\varepsilon_2, h)=\{g\in G:\exists\alpha=[o,go],s.t.~
\alpha_{[\varepsilon_1,\varepsilon_2]}~\text{is}~(\nu,h)\text{-barrier-free}\}.
\end{equation}

\begin{Lemma}\label{lemma:con}
Fix $\rho\in (0, 1]$, and choose any $\varepsilon_1<\varepsilon_2\in (0, \rho)$ so that $\varepsilon_2\rho\in (\varepsilon_1,\varepsilon_2)$.  Let $h$ be any element. If our group action satisfies the DOP condition  and PEG, then $V(\varepsilon_1,\varepsilon_2, h)$ is negligible in $(\rho, \Delta)$-annuli in $G$.

Moreover, if the action is SCC, then $V(\varepsilon_1,\varepsilon_2, h)$ is exponentially negligible in $G$.
\end{Lemma}

\begin{proof}
By Lemma \ref{LargeAnnulus}, it suffices to prove that $\displaystyle\frac{|V(\varepsilon_1,\varepsilon_2, h)\cap A([\rho n,n])|}{|A([\rho n,n])|}\rightarrow 0$ as $n\rightarrow\infty$.
Let $g\in V(\varepsilon_1,\varepsilon_2, h)\cap A([\rho n,n])$ and denote  $|g|=k$, so $\rho n\le k \le n$. By definition of $V(\varepsilon_1,\varepsilon_2, h)$, there exists a geodesic $\alpha=[o,go]$, so that $\alpha([\varepsilon_1k,\varepsilon_2k])$ is $(\nu,h)$-barrier-free.
Set $x=\alpha(\varepsilon_1 n),y=\alpha(\varepsilon_2 \rho n)$. By the choice of $\varepsilon_2 \rho\in (\varepsilon_1,\varepsilon_2)$, we see that $[x, y]_\alpha=\alpha([\varepsilon_1 n,\varepsilon_2 \rho n])$ is a subsegment of $\alpha([\varepsilon_1k,\varepsilon_2k])$, and thus is $(\nu,h)$-barrier-free.

We now subdivide our discussion into three cases, the first two of which could be viewed degenerate cases of  the third one. However, we treat them separately in order to illustrate the idea of the latter one.

\textbf{Case 1.}\,\,Assume that $x,y\in N_M(Go)$ so there exists $u,v\in G$ such that $$d(x,uo)\leq M,d(y,vo)\leq M.$$ Thus, $[x,y]_{\alpha}$ is a $(\nu,h)$-barrier-free geodesic between $B(uo,M)$ and $B(vo,M)$. So $u^{-1}v\in \mathcal{V}_{\nu, h}$.

Denote $\varepsilon=\varepsilon_2 \rho-\varepsilon_1>0$.  Since $d(x,y)=\varepsilon n$ and $|d(uo,vo)-d(x,y)|\leq 2M$, we have $$u^{-1}v\in \mathcal{V}_{\nu, h}( \varepsilon n,2M).$$ Clearly we have, $$u\in A(\varepsilon_1n,M),\;\; v^{-1}g\in A(k-\varepsilon_2 \rho n,M).$$
 Therefore, setting $\Delta=2M$, we obtain that $g=u(u^{-1}v)(v^{-1}g)$ lies the following set
\begin{equation}\label{Case1Set}
A(\varepsilon_1 n,\Delta)\cdot\mathcal{V}_{\nu, h}(\varepsilon n,\Delta)\cdot A(k-\varepsilon_2 \rho n, \Delta).
\end{equation}

\textbf{Case 2.}\,\, Assume that one of $\{x, y\}$ lies outside $N_M(Go)$. Let's assume first that $x\in N_M(Go),y\notin N_M(Go)$, so there exists $u\in G$ such that $d(x,uo)\leq M$. Consider the maximal open segment $(y_1,y_2)$ of $\alpha$ which contains $y$ but lies outside $N_M(Go)$. Hence, there exists $v_1,v_2\in G$ such that $d(y_i,v_io)\leq M$ for $i=1,2$. By definition, we have $v_1^{-1}v_2\in \mathcal{O}_M$.

Set $s=d(o,y_1)\in[\varepsilon_1 n,\varepsilon_2 \rho n],t=d(o,y_2)\in[\varepsilon_2 \rho n,k]$, where $n \ge k\ge \rho n$. Thus, $d(y_1,y_2)=t-s$, and $|d(v_1o,v_2o)-(t-s)|\leq2M\leq\Delta$. This means that $$v_1^{-1}v_2\in \mathcal{O}_M(t-s,\Delta).$$
Similarly as above, we have that
\[u^{-1}v_1\in\mathcal{V}_{\nu, h}(s-\varepsilon_1 n,\Delta),~v_2^{-1}g\in A(k -t,\Delta).\]

Consequently,  the element $g=u(u^{-1}v_1)(v_1^{-1}v_2)(v_2^{-1}g)$ lies in the following set
\begin{equation}\label{Case2Set}
A(\varepsilon_1 n,\Delta)\cdot\mathcal{V}_{\nu, h}(s-\varepsilon_1 n,\Delta)\cdot\mathcal{O}_M(t-s,\Delta)\cdot A(k-t,\Delta)
\end{equation}
where $s\in[\varepsilon_1 n,\varepsilon_2 \rho n]$ and $t \in[\varepsilon_2 \rho n,k]$.

Similarly, when $x\notin N_M(Go)$ and $y\in N_M(Go)$, we obtain
\[g\in A(i,\Delta)\cdot\mathcal{O}_M(j-i,\Delta)\cdot\mathcal{V}_{\nu, h}(\varepsilon_2\rho n-j,\Delta)\cdot A(k-\varepsilon_2\rho n,\Delta),\]
where $i\in[0,\varepsilon_1n],j\in[\varepsilon_1n,\varepsilon_2\rho n]$.

 \textbf{Case 3.}\,\, We now consider the general case  that $x,y\notin N_M(Go)$. Recall that $\varepsilon=\varepsilon_2\rho-\varepsilon_1$.  By  Lemma \ref{lemma:out}, the set $U(\varepsilon)$ is negligible. Without loss of generality, we can assume that   $g\notin U(\varepsilon)$. This implies that  $[x,y]_\alpha\cap N_M(Go)\ne \emptyset$. Indeed, if not,  then the geodesic segment $[x,y]_\alpha$ lies outside $N_M(Go)$. Since $[x,y]_\alpha$ is a subsegment of $\alpha=[o,uo]$ of length $(\varepsilon_2\rho-\varepsilon_1)n$ outside $N_M(Go)$,   we obtain  $g\in U(\varepsilon_2\rho-\varepsilon_1)$, that is a contradiction.

Hence, consider the maximal open segments $(x_1,x_2)_\alpha,(y_1,y_2)_\alpha$ of $\alpha$ outside $N_M(Go)$ which contain $x,y$ respectively. Since $[x,y]_\alpha\cap N_M(Go)\ne \emptyset$, these two intervals are disjoint.

Denote $i=d(o,x_1),j=d(o, x_2)$ and $s=d(o, y_1),t=d(o,y_2)$. Then $i\in[0, \varepsilon_1 n], j< s\in [ \varepsilon_1 n, \varepsilon_2 \rho n], t\in [\varepsilon_2\rho n, k]$, where $k\in [\rho n, n]$.  By the same reasoning as in the previous two cases, we have
\begin{equation}\label{Case3Set}
g\in A(i,\Delta)\cdot\mathcal{O}_M(j-i,\Delta)\cdot\mathcal{V}_{\nu, h}(s-j,\Delta)
\cdot\mathcal{O}_M(t-s,\Delta)\cdot A( k-t,\Delta)
\end{equation}
for each $g\in V(\varepsilon_1,\varepsilon_2)\cap A([\rho n,n])$ with $|g|=k$.

Note that $\varepsilon_2\rho\in (\epsilon_1,\epsilon_2)$ and $k\in [\rho n, n]$. We look at the index set $$\Lambda=\{(i,j,s,t)\in\mathbb{N}^4:0\leq i\leq\varepsilon_1n\leq j\leq s\leq\varepsilon_2\rho n\leq t\leq n\},$$
over which, we define
$$
V_{(i,j),(s,t)}:=A(i,\Delta)\cdot\mathcal{O}_M(j-i,\Delta)\cdot\mathcal{V}_{\nu, h}(s-j,\Delta)
\cdot\mathcal{O}_M(t-s,\Delta)\cdot B_{n-t+\Delta}.
$$

Combining (\ref{Case1Set}), (\ref{Case2Set}) and (\ref{Case3Set}),  we have the following decomposition
\begin{equation} \label{eq:decom}
V(\varepsilon_1,\varepsilon_2)\cap A([\rho n,n])\subseteq  \bigcup_{(i,j,s,t)\in\Lambda} V_{(i,j),(s,t)},
\end{equation}
up to a negligible set $U(\varepsilon)$.

To conclude the proof, it remains to show that the right-hand set in (\ref{eq:decom}) is negligible. For that purpose, we consider a triple of lengths $(l_1, l_2, l_3)$ with $l_1+l_2+l_3=l\in[\varepsilon n,n]$. We observe that there are at most $(l_1+1)$ indexes $(i,j,s,t)\in\Lambda$ satisfying $j-i=l_1,s-j=l_2,t-s=l_3$. In fact, we can choose some $i \in [0, \varepsilon_1 n]$ first, and once $i$ is fixed, then $j, s, t$ are all determined by the triple $(l_1, l_2, l_3)$. However, the choice of $i$ can only change from $\varepsilon_1   n-l_1$ to $\varepsilon_1n$, so we have at most $l_1+1$ many $(i,j,s,t)\in\Lambda$ falling in the same triple $(l_1, l_2, l_3)$.

For each $V_{(i,j), (s,t)}$ with $j-i=l_1,s-j=l_2,t-s=l_3$, we  have the following estimate:
\begin{align*}
|V_{(i,j),(s,t)}|& \leq |A(i,\Delta)|\cdot|\mathcal{O}_M(j-i,\Delta)|\cdot|\mathcal{V}_{\nu, h}(s-j,\Delta)|
\cdot|\mathcal{O}_M(t-s,\Delta)|\cdot|B_{n-t+\Delta}| \\
&\prec  \exp(i\delta_G)\cdot |\mathcal{O}_M(l_1,\Delta)|\cdot|\mathcal{V}_{\nu, h}(l_2,\Delta)|\cdot|\mathcal{O}_M(l_3,\Delta)|
\cdot  \exp((n-l_1-l_2-l_3-i)\delta_G) \\
&\prec  \exp(n \delta_G) \cdot |\mathcal{O}_M(l_1,\Delta)|\cdot|\mathcal{V}_{\nu, h}(l_2,\Delta)|\cdot|\mathcal{O}_M(l_3,\Delta)|
\exp((-l_1-l_2-l_3)\delta_G),
\end{align*}
where we used $|B_{n-t+\Delta}| \asymp \exp((n-t)\delta_G)$ since the action has purely exponential growth.

Since the indexes $(i,j,s,t)\in\Lambda$ can be grouped according to the triple $(l_1, l_2, l_3)$, we obtain
\begin{align*}
&\frac{\sum_{(i,j,s,t)\in\Lambda} |V_{(i,j),(s,t)}|}{\exp(n\delta_G)} \\
\leq &    \sum_{\substack{\varepsilon n\leq l\leq n}}^{l_1+l_2+l_3=l} (l_1+1)|\mathcal{O}(l_1,\Delta)|\cdot|\mathcal{V}_{\nu, h}(l_2,\Delta)|\cdot|\mathcal{O}(l_3,\Delta)|
\exp((-l_1-l_2-l_3)\delta_G).
\end{align*}
This tends $0$ as $n\to \infty$ by Lemma \ref{lemma:0}(2). We conclude that   $V(\varepsilon_1,\varepsilon_2, h)\cap A([\rho n,n])$ is negligible. When the action is SCC, the above inequality tends to $0$ exponentially fast. The proof of the result is complete.
\end{proof}

\subsection{Negligible pairs of elements}
The goal of Theorem \ref{thm:gen} is to show a random pair $(u_1, u_2)\in G^{(2)}$ generates a free group of rank $2$. We now define two negligible sets of 2-tuples $(u_1, u_2)\in G^{(2)}$, whose properties shall fail to be a free basis.

For any $u\in G$, let $\alpha=[o,uo]$ be any geodesic with length parametrization $\alpha(t)$. Define $\overline{\alpha}=[o,u^{-1}o]$ to be the geodesic with parametrization  $\overline{\alpha}(t):=u^{-1}\alpha(|u|-t)$.

Given $0<\varepsilon_1<\varepsilon_2<1$ and $C>0$, let $Z(\varepsilon_1,\varepsilon_2)$ be the set of $u\in G$ such that for some   $\alpha=[o,uo]$,  one of the following holds:
\begin{enumerate}
\item
$\overline{\alpha}$ intersect the $C$ neighborhood of the subsegment $\alpha_{[\varepsilon_1,\varepsilon_2]}$ of $\alpha$ or
\item
  $\alpha$ intersect the $C$ neighborhood of the subsegment ${\overline{\alpha}}_{[\varepsilon_1,\varepsilon_2]}$ of $\overline{\alpha}$.
\end{enumerate}

 In other words,
 \begin{equation}\label{Zepsilon12}
Z(\varepsilon_1,\varepsilon_2, C)=\{u\in G:\exists\alpha=[o,uo], s.t.~\overline{\alpha}\cap N_C(\alpha_{[\varepsilon_1,\varepsilon_2]})\neq \emptyset~\text{or}~
\alpha\cap N_C({\overline{\alpha}}_{[\varepsilon_1,\varepsilon_2]})\neq \emptyset\}.
\end{equation}

\begin{Lemma} \label{lemma:sep1}
Let    $0<\varepsilon_1<\varepsilon_2\leq1-\varepsilon_1<\rho<1$ and $C>0$. If our group action satisfies the DOP condition  {and purely exponential growth}, then $Z(\varepsilon_1,\varepsilon_2, C)$ is negligible in $(\rho, \Delta)$-annuli in $G$. Moreover, if the action is SCC, then $Z(\varepsilon_1,\varepsilon_2, C)$ is exponentially negligible in $G$.
\end{Lemma}

\begin{proof}
For any $u\in Z(\varepsilon_1,\varepsilon_2)\cap A([\rho n,n])-U(\frac{\varepsilon_1}{8})$, there exists $\alpha=[o,uo]$ satisfying the condition in the definition of $Z(\varepsilon_1,\varepsilon_2)$. Denote $j=|u|$ and then $\rho n\leq j\leq n$.

Without loss of generality, assume that $\overline{\alpha}\cap N_C(\alpha_{[\varepsilon_1,\varepsilon_2]})\neq \emptyset$. By definition, there exists $\varepsilon_1j\leq i\leq\varepsilon_2j$, so that $\overline{\alpha}\cap N_C(\alpha(i))\neq\emptyset$. Thus, there exists $s\in[i-C,i+C]$ such that  $$d(\overline{\alpha}(s),\alpha(i))\leq C.$$
Set $x=\alpha(i),y=\overline{\alpha}(s)=u^{-1}\alpha(j-s)$. Thus, $d(x, y)\le C$.

We follow a  similar analysis as in the proof of Lemma \ref{lemma:con}.

\textbf{Case 1.}\,\,Assume that $x,y\in N_M(Go)$, so there exist $v,w\in G$ such that $$d(x,vo)\leq M,\; d(y,wo)\leq M.$$
This implies $d(vo, wo)\le d(x, y)+2M\le 2M+C$, so $v^{-1}w\in B_{2M+C}$. Since $\overline \alpha(0)=o$ and $s\in[i-C,i+C]$,  we have $d(o, y)=s$ and then $|d(o, wo)-s|\le M+C$. Thus,
$$v\in A(i,M), w\in A(i,M+C).$$

We  can now write $u=v(v^{-1}uw)(w^{-1} v)v^{-1}$, where
\begin{align*}
d(uwo, vo)&\le d(o, uy)-d(o, x)+2M\\
&\le d(o, \alpha(j-s))-d(o, \alpha(i))+2M\\
&\le j-s-i+2M\le j-2i+C.
\end{align*}
which implies $v^{-1}uw\in A(j-2i,2M+C)$.

Noting that $v\in A(i,M)$, the set of elements $u$ in this case belongs to the following set
$$
A(i,M)\cdot A(j-2i,2M+C)\cdot B_{2M+C} \cdot A(i,M).
$$

\textbf{Case 2.}\,\,Assume that one of $\{x, y\}$ lies outside $N_M(Go)$. For definiteness, assume that $x\in N_M(Go),y\notin N_M(Go)$; the other case is symmetric. Then there exists $v\in G$ such that $d(x,vo)\leq M$. Consider the maximal open segment $(y_1,y_2)$ of $\overline{\alpha}$ which contains $y$ but lies outside $N_M(Go)$. Hence, there exists $w\in G$ such that $d(y_1,wo)\leq M$.

Since $u\notin U(\frac{\varepsilon_1}{8})$ is assumed and then $u^{-1}\notin U(\frac{\varepsilon_1}{8})$ by definition, we obtain that $d(y_1,y_2)< \frac{\varepsilon_1}{8}j$. Thus we have $d(y_1,y)\leq d(y_1,y_2)< \frac{\varepsilon_1}{8}j$. This yields
$$d(vo,wo)\leq d(vo, x)+d(x, y)+d(y, y_1)+d(y_1, wo)\le  \frac{\varepsilon_1}{8}j+2M+C.$$

Hence, we can also write $u=v(v^{-1}uw)(w^{-1} v)v^{-1}$, where
\[v\in A(i,M),v^{-1}w\in B_{\frac{\varepsilon_1}{8}j+2M+C},
v^{-1}uw\in B_{j-2i+\frac{\varepsilon_1}{8}j+2M+C}.\]

\begin{comment}Symmetrically, if $x\notin N_M(Go),y\in N_M(Go)$. We can write $u=v(v^{-1}uw)w^{-1}$, where
\[v\in B_{i+M+C},w\in A(i,M),v^{-1}w\in B_{\frac{\varepsilon_1}{8}j+2M+C},
v^{-1}uw\in B_{j-2i+\frac{\varepsilon_1}{8}j+2M+C}.\]
\end{comment}

\textbf{Case 3.}\,\, Assume $x,y\notin N_M(Go)$.
Consider the maximal open segment $(x_1,x_2)_\alpha$ (resp. $(y_1,y_2)_\alpha$) of $\alpha$ (resp. $\overline{\alpha}$) which contains $x$ (resp. $y$) but lies outside $N_M(Go)$. Then there exist $v,w\in G$ such that $d(x_1,vo)\leq M,d(y_1,wo)\leq M$.  Similar argument as above we have the following conclusion: we  can write  $$u=v(v^{-1}uw)(w^{-1} v)v^{-1},$$ where
\[v \in B_{i+M+C},v^{-1}w\in B_{\frac{\varepsilon_1}{4}j+2M+C},
v^{-1}uw\in B_{j-2i+\frac{\varepsilon_1}{4}j+2M+C}.\]

Set $\Delta=2M+C$.
Summarizing the above three cases, we have
\begin{align*}
&|Z(\varepsilon_1,\varepsilon_2)\cap A([\rho n,n])\setminus U(\frac{\varepsilon_1}{8})|\\
\leq & 2\sum_{j=\rho n}^{n}\sum_{i=\varepsilon_1j}^{\varepsilon_2j} |B_{i+\Delta}|\cdot
|B_{j-2i+\frac{\varepsilon_1}{4}j+\Delta}|\cdot|B_{\frac{\varepsilon_1}{4}j+\Delta}|\\
\prec &  (1-\rho)n\cdot (\varepsilon_2-\varepsilon_1)n\cdot \exp((1-\frac{\varepsilon_1}{2})n\delta_G)
\end{align*}
where the last line used
$$|B_{i+\Delta}|\cdot
|B_{j-2i+\frac{\varepsilon_1}{4}j+\Delta}|\cdot|B_{\frac{\varepsilon_1}{4}j+\Delta}| \asymp \exp(\delta_G(j+\frac{\varepsilon_1}{2}j -i))$$
which follows from the purely exponentially growth.

This shows that $Z(\varepsilon_1,\varepsilon_2, C)\cap A([\rho n,n])\setminus U(\frac{\varepsilon_1}{8})$ is negligible.
By Lemma \ref{lemma:out},  $U(\frac{\varepsilon_1}{8})$ is negligible, Thus the conclusion follows.
\end{proof}

Fix $0<\varepsilon_1<\varepsilon_2<\rho<1$ and $C>0$. Let $T(\varepsilon_1,\varepsilon_2, C)$ be the set of $(u_1,u_2)\in G\times G$ with the following property:

there exist two geodesics $\alpha:=[o,u_1o], \beta:=[o,u_2o]$ such that neither of them disjoint the $C$ neighborhood of the $[\varepsilon_1,\varepsilon_2]$-interval of the other.  In other words,
\begin{equation}\label{Tepsilon12}
T(\varepsilon_1,\varepsilon_2, C)= \left\{ \begin{array}{ll}
(u_1,u_2)\in G\times G:&\exists\alpha:=[o,u_1o], \beta:=[o,u_2o], s.t.~ \alpha\cap N_C(\beta_{[\varepsilon_1,\varepsilon_2]})\neq \emptyset~\\
& \text{or}~\beta \cap N_C(\alpha_{[\varepsilon_1,\varepsilon_2]})\neq \emptyset.
\end{array} \right\}.
\end{equation}

\begin{Lemma} \label{lemma:sep}
For any $0<\varepsilon_1<\varepsilon_2\leq 1-\varepsilon_1<\rho<1$ and $C>0$, if our group action satisfies the DOP condition and PEG condition, then $T(\varepsilon_1,\varepsilon_2, C)$ is negligible in $(\rho, \Delta)$-annuli in $G\times G$.

Moreover, if the action is SCC, then $T(\varepsilon_1,\varepsilon_2, C)$ is exponentially negligible in $G\times G$.
\end{Lemma}

\begin{proof}
Since the union of  two (exponentially) negligible sets is (exponentially) negligible, without loss of generality, we can assume that for all $(u_1,u_2)\in T(\varepsilon_1,\varepsilon_2, C)$, we have
$$\beta\cap N_C(\alpha_{[\varepsilon_1,\varepsilon_2]})\neq \emptyset.$$

Choose $1-\varepsilon_1<\rho<1$.
 By Lemma \ref{LargeAnnulus}, we can assume further that $(u_1,u_2)$ belongs to $T(\varepsilon_1,\varepsilon_2)\cap \Big(A([\rho n,n])\times A([\rho n,n])\Big)$.

Denote $n_1=|u_1|$.
By definition of $T(\varepsilon_1,\varepsilon_2, C)$, there exists $ i\in [\varepsilon_1n_1,\varepsilon_2n_1]$ so that $\beta\cap N_C(\alpha(i))\neq\emptyset$. Denote $x=\alpha(i)$ and $\Delta= C+2M$.  We proceed by a similar argument as before.

\textbf{Case 1.}\,\,Assume that $x\in N_M(Go)$ so there exists $v\in G$ such that $d(x,vo)\leq M$. Thus, $v\in A(i,M)$. Then $(u_1,u_2)$ can be written as $(v(v^{-1}u_1),v(v^{-1}u_2))$, where
\[v^{-1}u_1\in A(n_1-i,M),v^{-1}u_2\in A(n_2-i,C+M).\]

Note that $n_1\in [\rho n, n]$. In this case, we bound by above the number of elements $(u_1, u_2)$ as follows

\begin{align*}
&\leq \sum_{\substack{\rho n\leq n_1\leq n\\ \varepsilon_1n_1\leq i\leq\varepsilon_2n_1}}
|A(i,\Delta)|\cdot|A(n_1-i,\Delta)|\cdot|A([\rho n-i,n-i],\Delta)|\\
&\prec n\exp((2-\varepsilon_1)\delta_G n) = o(\exp(2\delta_G n)),
\end{align*}
so these pairs $(u_1, u_2)\in T(\varepsilon_1,\varepsilon_2)$ are exponentially negligible.

 \textbf{Case 2.}\,\,Otherwise, consider the maximal open segment $(x_1,x_2)_{\alpha^1}$ of $\alpha^1$, which contains $x$ but lies outside $N_M(Go)$. Denote $ j:= d(o, x_1),  l:=d(o, x_2)$. Thus $0\leq j  \le i$ and $i<l  \leq n_1$.

 \textbf{Subcase 2.1}\,\,$l-j\geq\frac{\varepsilon_1}{2}n_1$, then $u_1\in U(\frac{\varepsilon_1}{2})$. Since $U(\frac{\varepsilon_1}{2})$ is negligible in $G$ by Lemma \ref{lemma:out}, we have that $U(\frac{\varepsilon_1}{2})\times G$ is negligible as well in $G\times G$.

 \textbf{Subcase 2.2}\,\,$l-j<\frac{\varepsilon_1}{2}n_1$.
As before,  there exist $v_1,v_2\in G$ such that $d(x_1,v_1o)\leq M,d(x_2,v_2o)\leq M$. Thus, $v_1\in A(j,M)$.

Then $(u_1,u_2)$ can be written as $(v_1(v_1^{-1}v_2)(v_2^{-1}u_1),v_1(v_1^{-1}v_2)(v_2^{-1}u_2))$, where
\begin{align*}
&v_1^{-1}v_2\in A(l-j,2M),v_2^{-1}u_1\in A(n_1-l,M),\\
&v_2^{-1}u_2\in A((n_2-i)+(l-i),C+M)
\end{align*}

We consider the index set
\[\Lambda=\{(n_1,i,j,l)\in\mathbb{Z}^4:\rho n\leq n_1\leq n,\varepsilon_1n_1\leq i\leq\varepsilon_2n_1,0\leq j\leq i\leq l\leq j+\frac{\varepsilon_1}{2}n_1\}.\]
Hence, we have the upper bound on pairs $(u_1, u_2)$  of the second case as follows
\begin{align*}
&\leq \sum_{(n_1,i,j,l)\in\Lambda} |A(j,\Delta)|\cdot|A(l-j,\Delta)|\cdot|A(n_1-l,\Delta)|
\cdot|A([\rho n+l-2i,n+l-2i],\Delta)|\\
&\prec \sum_{(n_1,i,j,l)\in\Lambda}   \exp((n+n_1+l-2i)\delta_G)\\
& \prec n^4\exp((2-\frac{\varepsilon_1}{2})n\delta_G) = o(\exp(2\delta_G n)).
\end{align*}
Therefore, in this case, we have proved the negligibility of $T(\varepsilon_1,\varepsilon_2)$. The proof is complete.
\end{proof}

\section{The proof of the Theorems}\label{SecProofs}

This section is devoted to the proof of  the theorems of this paper.

\subsection{Generically free subgroups}
Let $\Lambda>0$. Denote by $\mathcal{F}^{(k)}$ by the set of
$k$-tuples  \[\{u_1,\cdots,u_k\}\in G^{(k)}\] such that
\begin{enumerate}
\item
$\langle u_1,u_2,\cdots,u_k\rangle$ is a free group of rank $k$ consisting  of contracting elements except the identity,
\item
the map $h\in \langle u_1,u_2,\cdots,u_k\rangle \mapsto ho \in Y$ is a $(\Lambda, 0)$-quasi-isometrically embedded map.
\end{enumerate}

Let   $\mathbb{F}(u_1,\cdots,u_k)$ be the free group generated by the $k$-tuple  $\{u_1,\cdots,u_k\}$.
In order  to prove that $\mathcal{F}^{(k)}$ is generic in $G^{(k)}$, the  idea is to construct a generic subset $E\subseteq G^{(k)}$, such that for any $\overrightarrow{u}=(u_1,\cdots,u_k)\in E$ and any nontrivial freely reduced word $W\in\mathbb{F}(u_1,\cdots,u_k)$, we can construct an admissible path from $o$ to $Wo$ that satisfies the conditions of Proposition \ref{pro:quasi-geo} and thus the path is a quasi-geodesic by the same proposition. This then concludes the proof of Theorem \ref{thm:gen}.

To be clear, we fix some notations and constants at the beginning (the reader is encouraged to read the proof first and return here until the constant appears).

\vspace{1em}
\hspace{-1.5em}\textbf{Setup}\,\,
\begin{enumerate}
\item We denote by ${C}>0$ the contraction constant for the contracting system $\{g \ax(h):g\in G\}$.   Assume that $ {C}$ satisfies Lemma \ref{lemma:prop} as well.
\item  Let $\mathbb{X}=\{g N_{ {C}}(A): g\in G\}$ be the contracting system of Lemma \ref{lemma:strongcontracting} with contraction constant $\kappa$. We denote $X=N_{ {C}}(A)$.

\item If our group action satisfies DOP condition, then constants $\nu,M>0$ are given by definition of DOP condition and Proposition \ref{pro:neg}. Moreover if the group action is SCC, then $\nu,M>0$ are given by Proposition \ref{pro:purelyexp}.

\item   Let $D=D(\kappa,9C)>16C$ be the constant of admissible paths given by Proposition \ref{pro:quasi-geo}.
\item Take $m>0$ so that $|h^m|>D+2\nu$. This can be done since $h$ is a contracting element, then we have $n\in \mathbb{Z}\mapsto h^n\in G$ is a quasi-isometric embedding map of $\mathbb{Z}\rightarrow G$.

\end{enumerate}

We refer the reader to the definitions of the set $V(2\varepsilon,1-2\varepsilon, h^m)$ in (\ref{Vepsilon12}), the set $W(\varepsilon, C)$  in (\ref{Wepsilon}), the set $Z(\varepsilon,1-\varepsilon, C)$  in (\ref{Zepsilon12}) and the set $T(\varepsilon,1-\varepsilon, C)$  in (\ref{Tepsilon12}).
\begin{Lemma}\label{lemma:gen}
Fix $1>\rho>\frac{8}{9}$ and $\varepsilon\in(1-\rho,\frac{1}{4})$. The subset $E$ of all $\overrightarrow{u}=(u_1,\ldots,u_k)\in G^{(k)}$ satisfying the following conditions is generic.
\begin{enumerate}
\item $|u_i|\geq\rho|\overrightarrow{u}|$ for $1\leq i\leq k$.
\item $u_i^{\pm 1}\notin V(2\varepsilon,1-2\varepsilon, h^m)\cup W(\varepsilon, C)\cup Z(\varepsilon,1-\varepsilon, C)$ for $1\leq i\leq k$.
\item $(u_i^{\pm 1},u_j^{\pm 1})\notin T(\varepsilon,1-\varepsilon, C)$ for $i\neq j\in\{1,2,\ldots,k\}$.
\end{enumerate}
When the action is SCC, the set $E$ is exponentially generic.
\end{Lemma}

\begin{proof}
It suffices to show that  the set of $\overrightarrow{u}\in G^{(k)}$  in each statement as above  is generic. It is clear that our choice of $\rho,\varepsilon$ satisfy all the condition of the lemmas in Section \ref{sec:negligiblesubset}. Hence the assertion (1) is given by Lemma \ref{LargeAnnulus}. The  assertion (2) is a consequence of Lemmas \ref{lemma:con}, \ref{lemma:short} and   \ref{lemma:sep1} together. And the assertion (3) follows from Lemma \ref{lemma:sep}.
\end{proof}

Now we are ready to prove Theorem \ref{thm:gen}.

\begin{proof}[Proof of Theorem \ref{thm:gen}]
For notational simplicity, we give the proof for $k=2$.  Let $E$ be the subset of $G\times G$ provided by Lemma \ref{lemma:gen}. It suffices to show that $E$ in contained in $\mathcal{F}^{(2)}$.

Fix a choice $(u_1,u_2)\in E$. We choose a geodesic $\alpha=[o,u_1o]$, and denote by $\overline{\alpha}:=[o,u_1^{-1}o]$ a geodesic from $o$ to $u_1^{-1}o$. Similarly, we define $\beta=[o, u_2o]$ and its reverse $\overline\beta:=[o,u_2^{-1}o]$.  Denote $n_1=|u_1|$ and $n_2=|u_2|$.

Let $W$ be a non-trivial freely reduced word in $\mathbb{F}(u_1,u_2)$. We shall prove that the evaluation of the word $W$ in $G$ gives a non-trivial contracting element.  For this purpose,  we can assume without loss of generality that $W$ is cyclically reduced so that the bi-infinite word $W^\infty=\cdots W\cdot W\cdot W\cdots$  is reduced, written explicitly as $$W^\infty= \cdots x_{-1}x_0x_1x_2\cdots x_j\cdots$$ where each $x_j\in \{u_1,u_2,u_1^{-1},u_2^{-1}\}.$

Associated with the bi-infinite word $W^\infty$, we construct a bi-infinite path $\gamma$ as a concatenation of geodesic segments $\gamma_j$ for $j\in \mathbb Z$ as follows:
\[\gamma=\cdots\cdot\gamma^{-1}\cdot\gamma^{0}\cdot \gamma^{1}\cdot \gamma^{2} \cdots\cdot \gamma^{j}\cdots,\]
where $\gamma^0=[o, x_0 o]$ and for $j\ge 1$, $\gamma^j$ is the $x_1\cdots x_{j-1}$-translate of $\alpha$ or $\beta$ depending on $x_j$.

We now describe a procedure to convert the path $\gamma$ to be an admissible path  by truncating certain subpaths.

Since $u_1, u_2\notin V(2\varepsilon,1-2\varepsilon, h^m)$ defined in (\ref{Vepsilon12}), we have that each $\gamma^j$ contains a $(\nu,h^m)$-barrier in $\gamma^{j}_{[2\varepsilon, (1-2\varepsilon)]}$. Then there exists an element $g_j\in G$ such that
\[\max\{d(g_j o,\gamma^{j}_{[2\varepsilon, (1-2\varepsilon)]}),\:
d(g_ih^m o,\gamma^{j}_{[2\varepsilon,(1-2\varepsilon)]})\}\leq \nu\le C.\]

We denote $X=N_{ {C}}(A)$. Let $v_j,w_j$ be the entry and exit point of $\gamma^{j}$ into $N_C(g_jA)=g_j X$ respectively. We must have
$$
d(v_j, w_j)\ge d(o, h^mo)-2\nu.
$$

Since $u_1, u_2\notin W(\varepsilon, C)$ defined in (\ref{Wepsilon}), we have $d(v_j,w_j)\leq \varepsilon\min\{|u_1|, |u_2|\}\le \varepsilon   n$, and so $v_j,w_j\in \gamma^j_{[\varepsilon,(1-\varepsilon)]}$ follows from the definition of $W(\varepsilon, C)$. In other words, the subsegment $[v_j,w_j]_{\gamma^j}$ of $\gamma^j$ is contained in $g_jX$ by Lemma \ref{lemma:prop}.

We now truncate the subpath $[v_{j-1}, \gamma^j_-]\cdot [\gamma_-, w_j]$ from $\gamma$ and replace it with a geodesic. The resulting path is given as follows
\[\beta=\cdots [v_{-1},w_{-1}]_{\gamma^-1}\cdot [w_{-1}, v_0]\cdot [v_0,w_0]_{\gamma^0}\cdot [w_0, v_1]\cdot[v_1,w_1]_{\gamma^1}\cdot[w_1, v_2]\cdot\cdots
\cdot [w_{j-1}, v_j]\cdot[v_j,w_j]_{\gamma^j}\cdots,\]
where $[w_{j-1}, v_j]$ is a choice of geodesic between $w_{j-1}$ and $v_j$ for any $j\in \mathbb Z$.

By Lemma \ref{lemma:strongcontracting}, $\mathbb{X}=\{g N_{ {C}}(A): g\in G\}$ is a contracting system with bounded projection. In the following claim, we shall consider the admissible path  associated with $\{g_jX: j\in \mathbb Z\}$.

\vspace{.5em}
\hspace{-1.2em}\textbf{Claim}\,\, $\beta$ is an $(D, \tau)$-admissible path.

\begin{proof}[Proof of Claim]
First of all, we have $v_j,w_j\in g_jX$ and $d(v_j,w_j)\geq|h^m|-2\nu\geq D$, thus    the condition (\textbf{LL1}) is satisfied.

Recall that $W^\infty$ is a freely reduced word over $\{u_1, u_2, u_1^{-1}, u_2^{-1}\}$, so the pair of any two adjacent letters $(x_j, x_{j+1})$ does not belong to $Z(\varepsilon, 1-\varepsilon, 4C)$  and $T(\varepsilon, 1-\varepsilon, 4C)$. Since   $v_j,w_j\in \gamma^{j}_{[\varepsilon,(1-\varepsilon)]}$, we derive from   Lemma \ref{lemma:gen} that
\[\forall j\in \mathbb Z,\,\gamma^{j-1}\cap N_{4C}([v_j,w_j]_{\gamma^j})=\emptyset,\gamma^{j+1}\cap N_{4C}([v_j,w_j]_{\gamma^j})=\emptyset.\]

For simplicity, we write $X_j:=g_jX_j$.
Using Lemma \ref{lemma:notint}, we have
\begin{equation}\label{eq:notint1}
\forall  j\in \mathbb Z,\,\gamma^{j-1}\cap X_j=\emptyset, \gamma^{j+1}\cap X_j=\emptyset.
\end{equation}

Thus, by Lemma \ref{lemma:proj} we obtain $\|\Pi_{X_j}([w_{j}, \gamma^{j}_+])\|\leq 3 {C},\|\Pi_{X_j}(\gamma^{j+1})\|\leq3 {C}$.

For any $ j\in \mathbb Z$, we have
\begin{align*}
\|\Pi_{X_j}([w_j, v_{j+1}])\|
&\stackrel{(\ref{eq:pointproj})}{\leq}\|\Pi_{X_j}(\{w_j,v_{j+1}\})\|+ {C} \\
&\leq \|\Pi_{X_j}(\{w_j,\gamma^j_+\})\|+\|\Pi_{X_j}(\{\gamma^j_+,v_{j+1}\})\|+ {C}\\
&\stackrel{(\ref{eq:pointproj})}{\leq} \|\Pi_{X_j}([\{z_j,\gamma^j_+\}]_{\gamma^j})\|+\|\Pi_{X_j}(\gamma^{j+1})\|+3 {C}\\
&\leq 6 {C}+3 {C} \le 9C.
\end{align*}
A similar estimate as above shows
\[\|\Pi_{X_j}([w_{j-1}, v_{j}])\|\leq 9C.\]
Thus the condition (\textbf{BP}) is satisfied.

From (\ref{eq:notint1}), we have $X_j\neq X_{j+1}$ for all $j \in \mathbb Z$. Then all conditions in the definition of admissible paths are verified. Thus, $\beta$ is a $(D, 9C)$-admissible path.
\end{proof}

By Proposition \ref{pro:quasi-geo}, we know that $\beta$ is a $(\Lambda,0)$-quasi-geodesic and it is contracting. Thus, every non-trivial freely reduced word gives a non-trivial contracting element  so  $\langle u_1,u_2\rangle$ is a free group of rank 2.

This implies that $\langle u_1,u_2\rangle$ generates a free group of rank 2  consisting of contracting elements  such that the orbital map is $(\Lambda, 0)$-quasi-isometrically embedded. This concludes the proof of Theorem \ref{thm:gen}.
\end{proof}

\begin{proof}[Proof of Theorem \ref{thm:gen2}]
If a non-elementary group $G$ admit a proper SCC action on $(Y,d)$ with a contracting element, then the corresponding set defined in Lemma \ref{lemma:gen} is exponentially generic since these sets provided in Section \ref{sec:negligiblesubset} are exponentially negligible. Therefore,  $\mathcal{F}^{(k)}$ is exponentially generic in $G^{(k)}$.
\end{proof}

\subsection{Statistical hyperbolicity}

\begin{proof}[Proof of Theorem \ref{thm:stahyper} for Annuli Case]
Choose any $0<\varepsilon <1$. Let
\[E=U(\frac{\varepsilon}{2}, C)\cup V(2\varepsilon,3\varepsilon,h^m)\cap W(\varepsilon,C).\]
Then By Lemma \ref{lemma:out}, \ref{lemma:con} and \ref{lemma:short} together, we have $\lim_{n\rightarrow+\infty}\frac{|E\cap A(n,\Delta)|}{|A(n,\Delta)|}=0$ for some $\Delta>0$. Now we fix such a $\Delta$.

For any $x\in A(n,\Delta)\setminus E$, we fix a geodesic $\alpha=[o,xo]$, and consider the following set
\[K_x=\{z\in A(n,\Delta):\exists\beta=[o,zo],s.t.\,\beta\cap N_{4C}(\alpha_{[\varepsilon,4\varepsilon]})\neq\emptyset\}.\]

We shall show that $K_x$ is negligible:
\begin{equation}\label{KxEQ}
\lim_{n\rightarrow \infty}\frac{|K_x|}{|A(n,\Delta)|}=0
\end{equation} for each $x$. Set $n_1=|x|$, then $n-\Delta\leq n_1\leq n+\Delta$.
We carry out the same analysis as in the proof of Lemma \ref{lemma:sep} to bound $|K_x|$: given any element $z\in K_x$ of length $n_2$, if $\alpha(i)$ intersect $N_M(Go)$ for some $\varepsilon n_1\leq i\leq4\varepsilon n_1$, then we can write $(x,z)$ as $(v_1(v_1^{-1}x),v_1(v_1^{-1}z))$ for some $v_1\in A(i,M)$, such that $v_1^{-1}z\in A(n_2-i,2\Delta)$; otherwise we can write $(x,z)$ as $(v_2(v_2^{-1}x),v_2(v_2^{-1}z))$ for some $i\leq l\leq i+\frac{\varepsilon}{2}n_1$ and some $v_2\in A(l,M)$, such that $v_2^{-1}z\in A((n_2-i)+(l-i),2\Delta)$ (by our choice of $x\notin U(\frac{\varepsilon}{2})$, subcase 2.1 of Lemma \ref{lemma:sep} can not happen).

If we introduce the index sets
\[\Lambda_1=\{(n_2,i)\in\mathbb{Z}^2:n-\Delta\leq n_2\leq n+\Delta,\varepsilon n_1\leq i\leq4\varepsilon n_1\},\]
\[\Lambda_2=\{(n_2,i,l)\in\mathbb{Z}^3:\rho n\leq n_2\leq n,\varepsilon n_1\leq i\leq4\varepsilon n_1,i\leq l\leq i+\frac{\varepsilon}{2}n_1\},\]
then we have
\begin{align*}
|K_x|&\leq \sum_{(n_2, i)\in \Lambda_1}|A(n_2-i,2\Delta)|+\sum_{(n_2,i,l)\in \Lambda_2}|A(n_2+l-2i,2\Delta)| \\
&\prec \exp((1-\varepsilon )n\delta_G)+ n^3\exp((1-\frac{\varepsilon}{2})n\delta_G),
\end{align*}
which implies (\ref{KxEQ}) from $|A(n,\Delta)| \asymp \exp(\delta_G n)$.

The next step is to bound the distance between $xo$ with the orbit point $yo$ outside $K_x$.
\vspace{.5em}
\hspace{-1.2em}\textbf{Claim}\,\,For any $y\in A(n, \Delta)\setminus K_x$, we have $d(x,y)\geq 2(n-4\varepsilon n-4\varepsilon\Delta-\Delta-4C)$.

\begin{proof}[Proof of Claim]
Since $x\notin V(2\varepsilon,3\varepsilon,h^m)$   in (\ref{Vepsilon12}), $\alpha$ contains a $(\nu,h^m)$-barrier in $\alpha_{[2\varepsilon,3\varepsilon]}$, so there exists an element $g\in G$ such that
\[\max\{d(g o,\alpha_{[2\varepsilon, 3\varepsilon3]}),\:
d(gh^m o,\alpha_{[2\varepsilon,3\varepsilon]})\}\leq \nu\le C.\]
We denote $X=N_C(A)$. Let $v,w$ be the entry and exit point of $\alpha$ into $gX$ respectively, so that $$d(v,w)\geq d(o,h^mo)-2\nu>D.$$ Since $x\notin W(\varepsilon,C)$ in (\ref{Wepsilon}), this implies that $v,w\in\alpha_{[\varepsilon,4\varepsilon]}$. Thus, $d(o, w)\ge \varepsilon n_1$ and $d(w, xo) \ge (1-4\varepsilon)n_1$.

For any $y\in A(n, \Delta)\setminus K_x$, we know from the definition of $K_x$ that for any  geodesic $\beta=[o,yo]$, $\beta\cap N_{4C}(\alpha_{[\varepsilon,4\varepsilon]})=\emptyset$. Thus, we have $\beta\cap gX=\emptyset$ by Lemma \ref{lemma:notint}.

If we choose $d(o,h^mo)-2\nu> D\ge16 {C}$ as in the setup, then for any $\gamma=[xo,yo]$, we have $\gamma\cap gX\neq\emptyset$. Indeed, if $\gamma\cap gX=\emptyset$, we will then  obtain a contradiction:
\begin{align*}
d(v,w)&\leq\|\Pi_X(\{v,1\})\|+\|\Pi_X(\{1,y\})\|+\|\Pi_X(\{y,x\})\|+\|\Pi_X(\{x,w\})\|\\
&\leq\|\Pi_X([1,v]_{\alpha})\|+\|\Pi_X(\beta)\|+\|\Pi_X(\gamma)\|+\|\Pi_X([w,x]_{\alpha})\|+4 {C}\\
&\leq 12 {C}+4 {C}<D
\end{align*}
from   a projection argument.

Let $u$ be the entry point of $\gamma$ in $gX$. Then $d(u,w)\leq 4C$ by the contracting property of $X$. Hence,
$$d(x o,u)\geq d(x o,w)-d(u,w)\geq n_1-4\varepsilon n_1-4C.$$ Since $y\in A(n,\Delta)$, we have $$n-\Delta\leq d(o,yo)\leq d(o,w)+d(w,u)+d(u,yo)\leq4\varepsilon n_1+4C+d(u,yo)$$ which yields
$$d(u,yo)\geq n-4\varepsilon n_1-\Delta-C.$$ We finally obtain
$$d(xo,yo)=d(xo,u)+d(u,yo)\geq2(n-4\varepsilon n-4\varepsilon\Delta-\Delta-4C)$$
concluding the proof of the claim.
\end{proof}

\vspace{.5em}
Let us return to the proof of the theorem. By the above claim, we have
\[\sum_{x,y\in A(n,\Delta)}d(xo,yo)\geq2(n-4\varepsilon n-4\varepsilon\Delta-\Delta-4C)\cdot(|A(n,\Delta)|-|E|)\cdot(|A(n,\Delta)|-|K_x|).\]
Notice $\lim_{n\rightarrow \infty}\frac{|K_x|}{|A(n,\Delta)|}=0$ and $\lim_{n\rightarrow \infty}\frac{|E|}{|A(n,\Delta)|}=0$. Then we obtain
\[\liminf_{n\rightarrow\infty}\frac{1}{|A(n,\Delta)|^2}\sum_{x,y\in A(n,\Delta)}\frac{d(xo,yo)}{n}\geq 2(1-4\varepsilon).\]
Since $\varepsilon$ is arbitrary,   we have $E_A(G,\Delta)=2$.
\end{proof}

\begin{proof}[Proof of Theorem \ref{thm:stahyper} for Ball Case] The proof is almost identical to that in annuli case. We only point out the difference in the proof.

Choose any $\frac{1}{2}<\rho<1$ and any $0<\varepsilon<\frac{\rho}{8}$. Let $E=U(\frac{\varepsilon}{2})\cup V(2\varepsilon,3\varepsilon)\cup W(\varepsilon)$. Then By Lemma \ref{lemma:out}, \ref{lemma:con} and \ref{lemma:short} together, we have $\lim_{n\rightarrow+\infty}\frac{|E\cap B_n|}{|B_n|}=0$.

For any $x\in A([\rho n,n])\setminus E$, set $n_1=|x|$, then $\rho n\leq n_1\leq n$. We fix a geodesic $\alpha=[o,xo]$ and consider
\[K_x=\{z\in A([\rho n,n]):\exists\beta=[o,zo],s.t.\,\beta\cap N_{4C}(\alpha_{[\varepsilon,4\varepsilon]})\neq\emptyset\}.\]

By the same argument as in annuli case, we have  $\lim_{n\rightarrow \infty}\frac{|K_x|}{|A([\rho n, n])|}=0$

\vspace{.5em}
\hspace{-1.2em}\textbf{Claim}\,\,For any $y\in A([\rho n,n])\setminus K_x$, we have $d(x,y)\geq 2\rho n-8\varepsilon n-8C$.

\begin{proof}[Proof of Claim]
The proof is the same as that in the annuli case, except that we now use the big annulus $A([\rho n,n])$. Note that
$$d(x o,u)\geq d(x o,w)-d(u,w)\geq n_1 -4\varepsilon n_1  -4C,$$
where $n_1\in [\rho n, n]$.
Since $y\in A([\rho n,n])$, we have
$$\rho n\leq d(o,yo)\leq d(o,w)+d(w,u)+d(u,yo)\leq4\varepsilon n_1+C+d(u,yo),$$
from which  we have $d(u,yo)\geq \rho n-4\varepsilon n_1-4C$. So $d(xo,yo)=d(xo,u)+d(u,yo)\geq2\rho n-8\varepsilon n-8C$.
\end{proof}

The same computation as above in annuli case gives
\[\liminf_{n\rightarrow\infty}\frac{1}{|B_n|^2}\sum_{x,y\in B_n}\frac{d(xo,yo)}{n}\geq 2\rho-8\varepsilon.\]
Since $\varepsilon$ can be made arbitrary small and $\rho$ can be arbitrary close to $1$, then we obtain $E_B(G)=2$.
\end{proof}

\begin{Examples}\label{FreeExample}
We carry out a concrete example to explain the convergence speed of $E_A(G, \Delta)=2$  of a statistical hyperbolic group is at most of order $O(n^{-1})$. Consider the free group $\mathbb{F}(a,b)$ and its Cayley graph with respect to the free generators $\{a,b\}$. It is easy to calculate
\[\frac{1}{|A(n,0)|^2}\sum_{x,y\in A(n,0)}d(x,y)=\frac{3}{4}\cdot 2n+\frac{1}{4}\frac{2}{3}\cdot(2n-2)+\frac{1}{4} \frac{1}{3}\frac{2}{3}\cdot(2n-4)+\cdots+0.\]
Thus we obtain
\[\big|\frac{1}{|A(n,0)|^2}\sum_{x,y\in A(n,0)}d(x,y)-2\big|=\big|\frac{n-1}{2n}-\frac{3}{4n}\cdot(\frac{1}{3}-\frac{1}{3^n})-\frac{1}{2}\big|=
\frac{1}{2n}+\frac{3}{4n}\cdot(\frac{1}{3}-\frac{1}{3^n})=O\big(\frac{1}{n}\big).\]
\end{Examples}

\end{document}